\documentclass[12pt]{article}
\usepackage[english]{babel}
\usepackage{amsmath,amssymb}
\usepackage[utf8]{inputenc}
\usepackage[T1]{fontenc}
\usepackage{tikz}
\usepackage{graphicx}
\graphicspath{{images/}}
\usepackage{url}
\usepackage{enumitem}
\usepackage{todonotes}

\usepackage{makeidx}
\makeindex
\usepackage{amsmath,amsfonts,amsthm,amssymb}
\usepackage{typearea}
\usepackage{graphicx}
\usepackage{geometry}
\theoremstyle{definition}
\newtheorem{definition}{Definition}
\theoremstyle{remark}
\newtheorem{Rem}[definition]{Remark}

\newtheoremstyle{mytheorem}{0.5cm}{0.2cm}{\slshape}{ }{\bfseries}{.}{ }{}
\theoremstyle{mytheorem}

\newtheorem{Prop}[definition]{Proposition}
\newtheorem{Thm}[definition]{Theorem}

\newtheorem{Lem}[definition]{Lemma}
\newtheorem{Def}[definition]{Definition}
\newtheorem{Cor}[definition]{Corollary}

\renewcommand{\P}{\mathbf{P}}

\DeclareMathOperator{\E}{{\bf E}}

\renewcommand{\epsilon}{\varepsilon}

\renewcommand{\phi}{\varphi}

\newlength{\querylen}
\usepackage{fancybox}

\begin{document}

\title{Migration-Contagion Processes}
\date{}
\author{F. Baccelli$^1$, S. Foss$^2$, and S. Shneer$^2$\\
1: INRIA and Telecom Paris, France. 2: Heriot--Watt University, Edinburgh, UK.}
\maketitle
\begin{abstract}
Consider the following migration process based on a closed network of $N$ queues with $K_N$ customers.
Each station is a $\cdot$/M/$\infty$ queue with service (or migration) rate $\mu$.
Upon departure, a customer is routed independently and uniformly at random to another station.
In addition to migration, these customers 
are subject to an SIS  (Susceptible, Infected, Susceptible) dynamics. That is,
customers are in one of two states: $I$ for infected, or $S$ for susceptible.
Customers can swap their state either from $I$ to $S$ or from $S$ to $I$ only in stations.
More precisely, at any station, each susceptible customer becomes infected
with the instantaneous rate $\alpha Y$ if there
are $Y$ infected customers in the station, whereas each infected customer recovers and
becomes susceptible with rate $\beta$. We let $N$ tend to infinity and assume that 
$\lim_{N\to \infty} K_N/N= \eta $, where $\eta$ is a positive
constant representing the customer density.
The main problem of interest is about the set of parameters of such a system for which
there exists a stationary regime where the epidemic survives in the limiting system.
The latter limit will be referred to as the thermodynamic limit. 
We establish several structural properties (monotonicity and convexity) of the system,
which allow us to give the structure of the phase transition diagram
of this thermodynamic limit w.r.t. $\eta$.
The analysis of this SIS model reduces to that of a wave-type PDE for which we found 
no explicit solution. This plain SIS model is one among several companion 
stochastic processes that exhibit both migration and contagion.
Two of them are discussed in the present paper as
they provide variants to the plain SIS model as well as some bounds and approximations.
These two variants are the DOCS (Departure On Change of State)
and the AIR (Averaged Infection Rate), which both admit closed-form solutions.
The AIR system is a classical mean-field model where the 
infection mechanism based on the actual population of infected customers
is replaced by a mechanism based on some empirical average of the number of infected customers
in all stations. The latter admits a product-form solution. DOCS features
accelerated migration in that each change of SIS state implies an immediate  departure. 
This model leads to another wave-type PDE that admits a closed-form solution.
In this text, the main focus is on the closed systems and their limits.
The open systems consisting of a single station with Poisson input are
instrumental in the analysis of the thermodynamic limits and are also of independent interest.
\end{abstract}

\tableofcontents

\section{Introduction}

The model described in the abstract aims at quantifying the role
of migration in the propagation of epidemics on the simplest possible
migration model, namely a closed network of $\cdot/M/\infty$ queues, and for the
simplest epidemic process, namely the SIS model.
The thermodynamic limit discussed in the abstract is considered in order to further simplify the problem.

The epidemic interpretation of the SIS thermodynamic model is as follows:
individuals move from place (station) to place where places are indexed by, say $\mathbb Z$.
The overall density of individuals is $\eta$ (i.e., the mean number of individuals
per station is $\eta$). An individual stays at a place for an independent
random time which is exponentially distributed with parameter $\mu$.
The last parameter can be seen as the {\em migration rate}.
The departure rate of individuals from a given place is hence $\lambda=\mu \eta$. 
When it leaves a place, an individual migrates to a place chosen
independently and 'at random'. At each station, individuals are subject to
the SIS dynamics with parameters $(\alpha,\beta)$. Infections and recoveries
take place in stations and conditionally on the state of the stations.

In order to answer the question of the abstract, we 
first consider the problem of a {\em single open} station of the M/M/$\infty$ type,
referred to as the $M/M/\infty$ {\em SIS reactor}.
The input features two types of customers, infected and susceptible. The customer states change
according to mechanisms as those described in the abstract.
The question is about the steady state of this queue.
Of particular interest to us is the proportion $p_o$ of infected customers
in the stationary output point process of this reactor in the steady state.
This system has its steady state with a joint generating function characterized
as a solution of a second order wave-type PDE for which we found no closed-form solution.
We nevertheless derive several structural properties of this open system,
which are of independent interest.

The connection of the SIS reactor with the problem of the abstract is the following:
Fix all parameters $\lambda,\mu,\alpha,$ and $\beta$ of the thermodynamic limit as defined
in the abstract.
If the latter has a non-degenerate steady state (namely a steady state with survival),
then there must exist a $0<p<1$ such that the open system with infected proportion $p$
in input has a proportion of infected in output, $p_o$, equal to $p$ in the steady state.
Conversely, if the SIS reactor has a stationary regime where
$p_o(p)=p$, for some $p$, then the thermodynamic limit has a stationary regime where
the epidemic persists for appropriate initial conditions.
In addition, it will be shown that if the SIS reactor admits no $p$ such that $p_o(p)=p$,
then for all initial conditions, when time tends to infinity,
the proportion of infected customers tends to 0 in the thermodynamic limit. 
It is why we study the conditions on the $\lambda,\mu,\alpha,$ and $\beta$ parameters
for such a $p$ to exist, a situation that we will call {\em survival},
as well as conditions under which no such $p$ exists, situation that we call {\em weak extinction}.

More precisely, we prove that the SIS thermodynamic limit admits the following phase 
transition diagram: fix all parameters other than $\eta$; there exists
a non-degenerate function $\eta_c(\alpha,\beta,\mu)$ such that, for
$\eta\le\eta_c$, there is strong extinction, whereas for $\eta>\eta_c$, there is survival.
We also derive bounds on the steady states of both the
single open station model and the thermodynamic model. For instance, 
we give bounds on $\eta_c$ and $p^*$ for the latter.
Some of these bounds are established under an hypothesis of negative correlation
which is substantiated by simulation but not proved at this stage.

All the structural results on the SIS reactor are proved.
All the results on the thermodynamic 
limit are proved under the assumption that the limits defining them
exist - which we conjecture.
This will be referred to as the {\em thermodynamic propagation of chaos ansatz}.

We also study the following variants of the SIS reactor:
\begin{itemize}
\item The DOCS reactor, where the infection rate is the same as above,
but where (a) an infection immediately leads to a departure, and (b)
a recovery immediately leaves to a departure.
This problem is simpler in that the associated PDE can be solved. 
\item The AIR reactor, where (a) the infection rate
of any susceptible is constant 
(rather than proportional to the actual number of infected), and (b)
the recovery mechanism is as in the SIS case. This last model is a product-form
Jackson network and admits a simple product-form solution.
\end{itemize}
In contrast to the plain SIS case, the stationary regimes of these two other open 
systems admit closed-form expressions. 
These two open loop models in turn lead to 
thermodynamic limits. More precisely, the DOCS reactor leads to the DOCS thermodynamic limit
where customers leave the station as soon as they change their state and are immediately
routed to another station chosen at random while keeping their new state,
and the AIR reactor leads to the AIR thermodynamic limit, which is a closed
network (similar to the plain SIS thermodynamic limit) where, in any station, susceptible customers
get infected with a rate that is proportional to the mean number of infected customers in all stations.
The closed-form solutions alluded to above are used to quantify 
the phase diagram of the last two thermodynamic limits.

The AIR system is conjectured to provide bounds to the plain SIS system.
These bounds are in the following sense.
Consider two thermodynamic limits $A$ and $A'$ with the same parameters $\alpha,\beta,\mu$ and with varying $\eta$.
System $A$ will be said to have {\em more infection in mean} than system $A'$
if the mean number of infected customers is more in $A$ than in $A'$ in steady state.
Assume that system $A$ (resp $A'$) admits a critical value $\eta_c$ (resp. $\eta_c'$) such that 
if $\eta>\eta_c$, (resp. $\eta>\eta_c'$),
then there is survival, whereas there is extinction otherwise.
System $A'$ will be said to be {\em safer} than system $A$ if $\eta'_c\ge \eta_c$.
If $A$ has more infection in mean than $A'$, then $A'$ is safer than $A$.
The converse is not true in general.
It will be proved that, under a certain negative correlation conjecture which
is backed by simulation, DOCS is safer than AIR.
There is numerical evidence that plain SIS can be safer than DOCS and conversely
depending on the parameters.

The paper is structured as follows:
A summary of the models (single station models and thermodynamic limit models) 
is given in Section \ref{sec:sum}, so as to make navigation between them easier.
The open SIS reactor is studied in Section \ref{sec2}.
Its stationary generating function satisfies a wave-type PDE for which
we could not find closed-form solutions so far.
We then establish rate conservation equations
which are useful throughout the paper. 
Section \ref{sec:SIDR} discusses the DOCS open reactor
and Section \ref{sec:airR} the AIR open reactor.
Both models can be solved in closed-form. The wave-type PDE of the DOCS reactor
has a closed-form solution, whereas the AIR reactor reduces to a product
form queuing network for which an explicit solution is already known.
The closed-form expressions provide bounds on the SIS dynamics.
Section \ref{sec:thermo} gathers results on the thermodynamic mean field limit
of the SIS dynamics. The analysis is based on structural
properties of the open SIS reactor. The main result is a characterization of the
structure of the phase diagram.
The same is done in Sections \ref{secAIRth} and \ref{secDOCSth} for the AIR and the
DOCS thermodynamic limits, respectively.
Section \ref{s:ttl} gives a probabilistic interpretation for the phase transition thresholds 
in terms of branching conditions.
Finally, Section \ref{ss:num} gathers additional numerical observations based on discrete event
simulation.

The appendix contains proofs and additional material. In particular,
Subsection \ref{sec4} in the appendix discusses inequalities that would follow from the property of
anti-association (negative correlation) of certain random variables in this SIS dynamics. 

We conclude this introduction by a brief overview of the relevant literature.
In the absence of mobility, the problem was extensively studied
in the particle system literature \cite{Liggett99}, where the model
is referred to as the {\em contact process}. This literature
contains a large corpus of results on the phase diagrams of
infinite graphs with finite degrees such as grids and regular trees.
The problem was also studied on finite deterministic graphs, where
the main question is that of the
separation between a logarithmic and an exponential growth of
the time till extinction with respect to the size of the graph.
There is also a series of studies of SIS epidemic models on population
partitioned into households and, in particular, on their correlation structure and
time to extinction, see e.g., \cite{Donnelli} and \cite{Britton} and references therein.

The SIS dynamics was also extensively studied on finite random graphs
For overviews on this class of questions, see, e.g. \cite{Pastor15}
and \cite{Newman15}. Moment closure techniques \cite{Kuehn16,Jrishnarajah} provide
an important tool in this context.  The contact process was also studied
on infinite random graphs with unbounded degrees.
The supercritical Bienaym\'{e}-Galton-Watson tree was studied in
\cite{Pemantle92} where it was shown that some critical values can
be degenerate. It was also extensively studied on
Euclidean point processes  \cite{Ganesan15, Hao18, Menard15}.

The analysis of the case with mobility is more recent.
The situation where agents perform a random walk on a
finite graph and agents meeting at a given point of the graph may infect each other
was studied in \cite{Figueiredo20}.
The situation where agents form a Poisson point processes and
migrate in the Euclidean plane was
studied in \cite{Baccelli22}, were
a computational framework based onmoment closure techniques
was proposed for evaluating the role of mobility on the propagation of epidemics.

The queuing model studied here may be seen as a discrete version of the model in
of \cite{Baccelli22}, or as a thermodynamic limit of that of \cite{Figueiredo20}.

\section{The Models}
\label{sec:sum}
\subsection{The Open Models}
All models in this subsection feature an open queuing system
with two types of customers. There are two independent Poisson
external arrival point processes: that of susceptible customers, with
intensity $\lambda q$, and that of infected customers, with intensity $\lambda p$, with $q=1-p$.

\subsubsection{The Plain SIS Reactor}
The SIS reactor features a single M/M/$\infty$ type station (that will be referred to as a reactor).
Service times of all customers in this queue are exponential with parameter $\mu$.
While waiting in the reactor, each susceptible customer becomes infected
with the instantaneous rate $\alpha n$ when there are $n$ infected customers, and
each infected customer recovers and becomes susceptible with rate $\beta$. It is because of this
interaction that we call the queue a reactor
(see the left part of Figure \ref{fig:fig1}).

\begin{figure}[!h]
\centering
\includegraphics[width=0.99\linewidth]{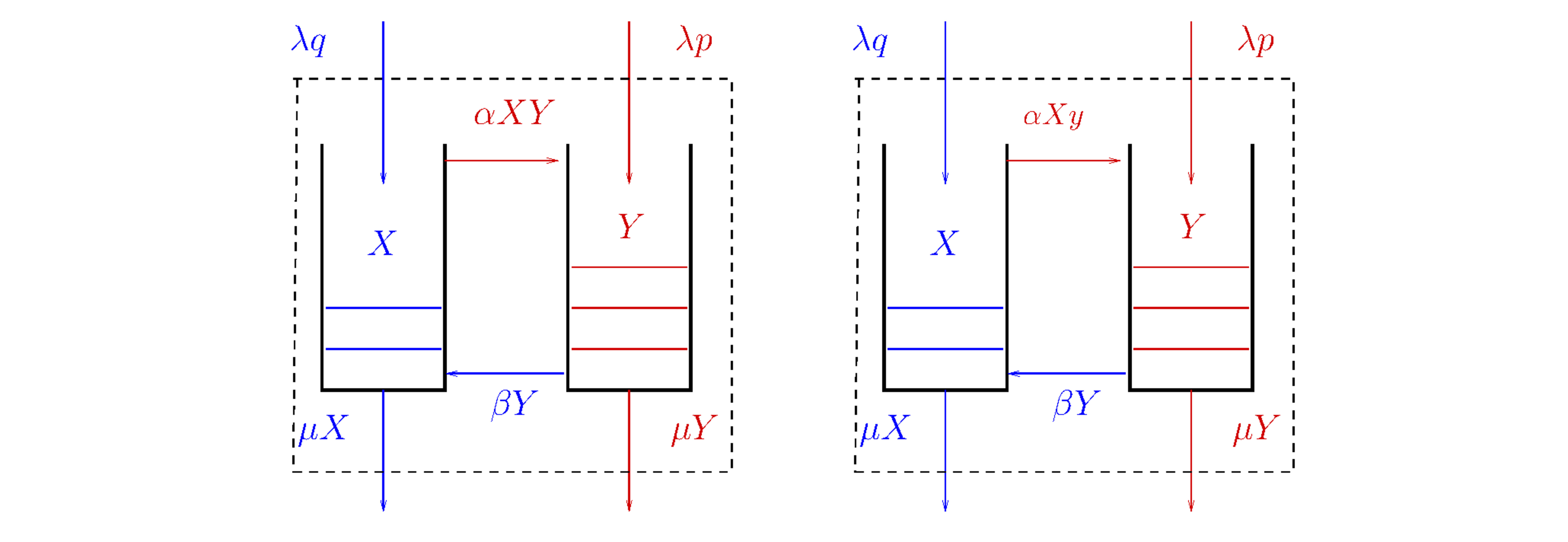}
\caption{On the left, the SIS reactor. On the right, the AIR reactor with parameter $y$.}
\label{fig:fig1}
\end{figure}

Let $X$ (resp. $Y$) denote the number of susceptible (resp. infected) 
customers in the steady state of the SIS reactor.
By classical arguments, one gets that for all $0\le x\le 1$, $0\le y\le 1$, the joint generating
function $\Phi(x,y)=\E[x^Xy^Y]$ satisfies the PDE
\begin{equation} 
\label{eq:sispde}
\left( \lambda q(1-x) + \lambda p(1-y)\right) \Phi(x,y)
= \mu (1-x)\Phi_x(x,y)+
\left(\mu (1-y)+ \beta(x-y)\right)\Phi_y(x,y) +
\alpha y (y-x)\Phi_{xy}(x,y) . \end{equation}

The term in $\Phi_{xy}(x,y)$, which comes from the infection rate in $XY$ allows
one to link this equation to the one dimensional wave equation.

\begin{Rem}
The 1 dimensional wave equation reads
$U_{tt} = c^2 U_{zz},$
with $c$ velocity, $t$ time, and $z$ space.
The last PDE leads (after a change of variables) to a relation between second derivatives of the form:
$\Psi_{xx}= \left(\frac y x\right)^2 \Psi_{yy}$ plus additional lower order terms.
That is, a velocity of $\frac y x$ when interpreting $x$ as time and $y$ as space.
In this sense, the wave equation satisfied
by the joint generating function involves a velocity that is not determined by the parameters
of the dynamics but only by the variables of the joint generating function. 
\end{Rem}

\subsubsection{The AIR Reactor}
The AIR model features a network of two M/M/$\infty$ stations.
All infected arrivals are routed 
to the second station (that of infected). The service rate in this station is $\nu=\mu+\beta$.
When a customer leaves this station, it leaves the network with probability
$\frac{\mu}{\mu +\beta}$ and is routed to the first station otherwise. 
All susceptible customers are routed to the first station (that of susceptible).
At time $t$, the service rate in this station is $\mu + y(t) \alpha$, where $y(t)$ is a positive
deterministic function that can be fixed at will.
When a customer leaves this station, it leaves the network with probability
$\frac{\mu}{\mu + y(t) \alpha}$ and is routed to the second station
with probability $\frac{y(t) \alpha}{\mu + y(t) \alpha}$.
This is depicted on the right part of Figure \ref{fig:fig1}.
The main difference with the SIS reactor is that the infection rate of a susceptible is
a state-independent deterministic function here. In particular, when
$y(t)$ is constant, we will not need the PDE here but will rather use the theory
of {\em product-form queuing networks} (see Section \ref{sec:airR}).

\subsubsection{The Averaging Mean Field of the SIS Reactor}
The {\em averaging mean-field limit} of the plain SIS model
is defined as the following limit of open networks. Consider a system with $K$ stations.
Each station is an M/M/$\infty$ queue with service rate $\mu$ and arrival rate $\lambda$.
Each arrival is independently declared infected with probability $p$ and susceptible otherwise.
In each station, an infected customer turns susceptible with rate $\beta$.
In each station a susceptible customer turns infected with rate
$$\frac \alpha  K \sum_{k=1}^K Y_k(t),$$
where $Y_k(t)$ is the number of infected nodes in station $k$ at time $t$.
So in this model, which is depicted on the left part of Figure \ref{fig:figthair},
conditionally on the state,
the infection rate of a susceptible in a station is proportional
to the {\em empirical mean} of the number of infected customers in all stations (rather
than to the number of infected customers in the same station in the SIS reactor).
The averaging mean-field limit of SIS is obtained when letting $K\to \infty$.
The empirical mean in question then converges
to a constant, which is also the mean number of infected customers in the typical station.
When it exists, this limit features a typical station which is a AIR model with the constraint that
$y(.)$ must be such that $y(t)=\mathbf{E}_{y(.)}[Y(t)]$ for all $t$, where $\mathbf{P}_{y(.)}$
is the distribution at time $t$ of the system with parameter $y(.)$, and $Y(t)$
is the number of infected customers at time $t$ in the AIR station.
The construction of such a system is discussed in Subsection \ref{ss:nhmc} in the appendix.
A single AIR station where $y(.)=\mathbf{E}_{y(.)}[Y(t)]$ will be referred to
as an {\em AIR-AMF} (AIR Averaging Mean-Field) reactor.

\subsubsection{The DOCS Reactor}
The DOCS model features a single station like in the SIS case.
The infection mechanism of the SIS model is replaced by a {\em simultaneous infection and departure}
mechanism with the following characteristics:
if the number of infected is $Y(t)$, each susceptible gets infected with rate $\alpha Y(t)$ and, 
upon infection, it immediately leaves the system for good.
There is also a ``natural'' departure rate of susceptible customers denoted by $\mu$. 
The total departure rate of infected is hence $\nu=\mu+\beta$
since when an infected customer recovers, it immediately
leaves the system. Equivalently, the total departure rate is $\mu+\beta$ and
upon departure, the infected customer keeps its state with probability $\mu/(\mu+\beta)$
or swaps to susceptible with the complementary probability. See the left part of Figure \ref{fig:fig11}

\begin{figure}[!h]
\centering
\includegraphics[width=0.99\linewidth]{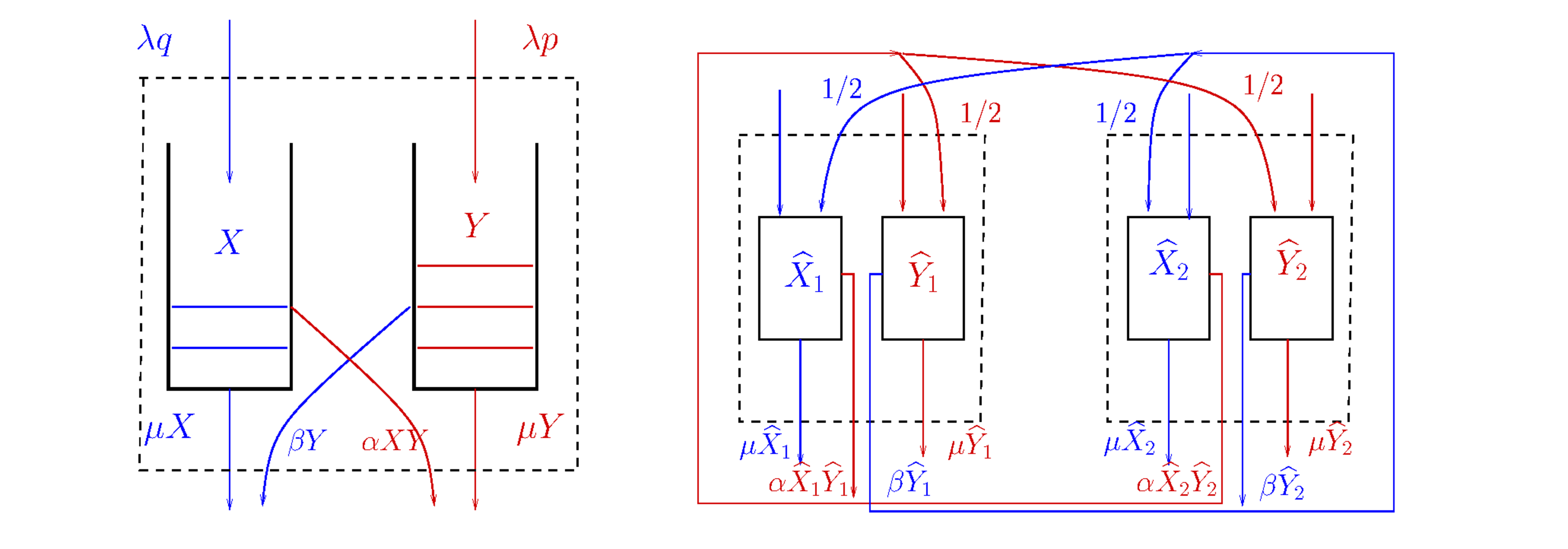}
\caption{On the left, the DOCS reactor. On the right, the routing DOCS network with two stations.}
\label{fig:fig11}
\end{figure}

In fact, if one just needs to determine $\mathbf E (X)$ in the DOCS reactor,
the state of departing customers does not matter. Therefore, we may (and will) analyze
a more general scheme where there are no recoveries but a general departure rate $\nu$
of infected and where $\nu$ and $\mu$ are not linked as above.

The associated PDE for the stationary generating function of the general DOCS reactor is
\begin{eqnarray} 
\label{eq:sipde}
\left( \lambda q(1-x) +\lambda p (1-y)\right) \Phi(x,y) 
= \mu (1-x) \Phi_x(x,y)
+ \nu (1-y) \Phi_y(x,y)
+ \alpha y (1-x)\Phi_{xy}(x,y) . 
\end{eqnarray}

\begin{Rem}
A wave type equation shows up in the DOCS and the Plain SIS systems,
whereas it does not in the AIR model.
This comes from the non-linear infection rate of the form $\alpha XY$ whenever
the state is $(X,Y)$. This leads to the 
$\frac{\partial}{\partial x}$ and  $\frac{\partial}{\partial y}$ terms in the corresponding PDEs,
which in turn lead to the corresponding wave equations. In contrast, in the AIR system,
this non-linearity is replaced by the linear rate $\alpha X\E[Y]$, with $\E[Y]$
seen as a parameter determined by a non-linear relationship, which explains
why the wave equation does not show up.
\end{Rem}

\subsection{The Routing Mean-Field Limit of the DOCS Reactor}
\label{ssec:siremeli}
The {\em DOCS routing mean-field limit}, which will be essential below,
is a variant of the DOCS reactor. It should not be confused with the DOCS thermodynamic
limit (defined in Subsection \ref{sec:ithl} below). This routing mean-field limit
is defined as follows. Consider first a finite system with $N$ stations
as depicted in the right part of Figure \ref{fig:fig11}.
The state variables in station $n$ are denoted by $\widehat X_n(t)$ and $\widehat Y_n(t)$.
Each station is an $\cdot$/M/$\infty$ station with service rate $\mu$ and external arrival rate $\lambda$.
Each external arrival is independently declared infected with probability $p$ and susceptible otherwise.
In each station, an infected customer turns susceptible with rate $\beta$.
But rather than staying in the station (as in SIS), this newly susceptible instantaneously leaves,
and rather than leaving for good as in the DOCS open reactor, this customer
joins another station chosen at random and independently among the $N$ stations.
Similarly, in station $n$, a susceptible customer turns infected with rate
$\alpha \widehat Y_n(t)$ at time $t$. 
But rather than staying in this station, the newly infected instantaneously 
joins another station chosen at random and independently among the $N$ stations.

When $N$ tends to infinity, we get a mean-field limit which is system with both external arrivals
and internal customer routing.
Again, the existence of this limit will not be discussed in the present paper. 
In the steady state of this limit, the overall
arrival point process of infected in a typical station
is Poisson with rate $\lambda p +\alpha \E[\widehat X\widehat Y]$ 
(sum of external and internal rates)
and that of susceptible is independent and Poisson with rate $\lambda q +\beta \E[\widehat Y]$
(sum of external and internal rates again).
Here $\widehat X$ and $\widehat Y$ represent the stationary state variables in this limit
(we use a hat on these variables to distinguish them from those of the DOCS reactor).

More generally, we will call {\em Routing DOCS reactor} a station 
which behaves as a DOCS reactor but where, in addition to an external infected (resp. 
susceptible) arrival
Poisson input point process of intensity $\lambda p(t)$ (resp. $\lambda q(t)$),
there is an additional ``re-routing'' infected (resp. susceptible)
Poisson point of intensity $\alpha \E[\widehat X(t)\widehat Y(t)]$ (resp.  $\beta \E[X(t)]$).

For all $0\le x\le 1$, $0\le y\le 1$, in steady state, the joint generating
function $\widehat\Phi(x,y)=\E[x^{\widehat X}y^{\widehat Y}]$ satisfies the PDE
\begin{eqnarray} 
\label{eq:routingpde}
& & \hspace{-2cm} \left( (\lambda q+\beta \E[\widehat Y])(1-x)
+(\lambda p +\alpha \E[\widehat X\widehat Y])(1-y)\right)\widehat \Phi(x,y) \nonumber\\
& = & \mu (1-x)\widehat \Phi_x(x,y)
 + 
\left(\mu + \beta \E[\widehat Y] \right)(1-y) \widehat \Phi_y(x,y) +   
\alpha y (1-x)\widehat \Phi_{xy}(x,y) . 
\end{eqnarray}
This PDE is an instance of the DOCS PDE in (\ref{eq:sipde}) with the following specific parameters:
$\lambda q$ is replaced by
$\lambda q+\beta \E[\widehat Y]$,
$\lambda p$ by
$\lambda p +\alpha \E[\widehat X\widehat Y]$, and
$\nu$ is taken equal to
$\mu + \beta$.

Here are a few observations on the routing DOCS models.
Consider first the model with $N$ stations. Consider the whole system as a single system.
Any customer has an exponential life time in the queue with parameter $\mu$. During its lifetime,
an individual changes state but stays in the system. Since
the total arrival rate to the system is $\lambda N$, the total number of customers in the stationary regime
is Poisson with parameter (and mean) $\lambda N/\mu$.
Due to the symmetry, the mean number of all customers in any station is hence $\lambda/\mu$.
Letting $N$ to tend to infinity and noting that uniform integrability holds,
we conclude that the mean number of customers stays equal to $\lambda/\mu$ in the infinite system too.
For all fixed $N$, in the stationary regime, the pairs of $(\widehat X,\widehat Y)$ vectors at
different stations form an exchangeable family of dependent random vectors,
with all coordinates summing up to a Poisson random variable. 

\subsection{The Thermodynamic Limits}
\label{sec:ithl}

\subsubsection{Definition}
The thermodynamic limits pertain to a family of closed networks as
illustrated for SIS in the left part of Figure \ref{fig:figthsis} or for DOCS in
the left part of Figure \ref{fig:figthdocs}.

\begin{figure}[h!]
\centering
\includegraphics[width=0.99\linewidth]{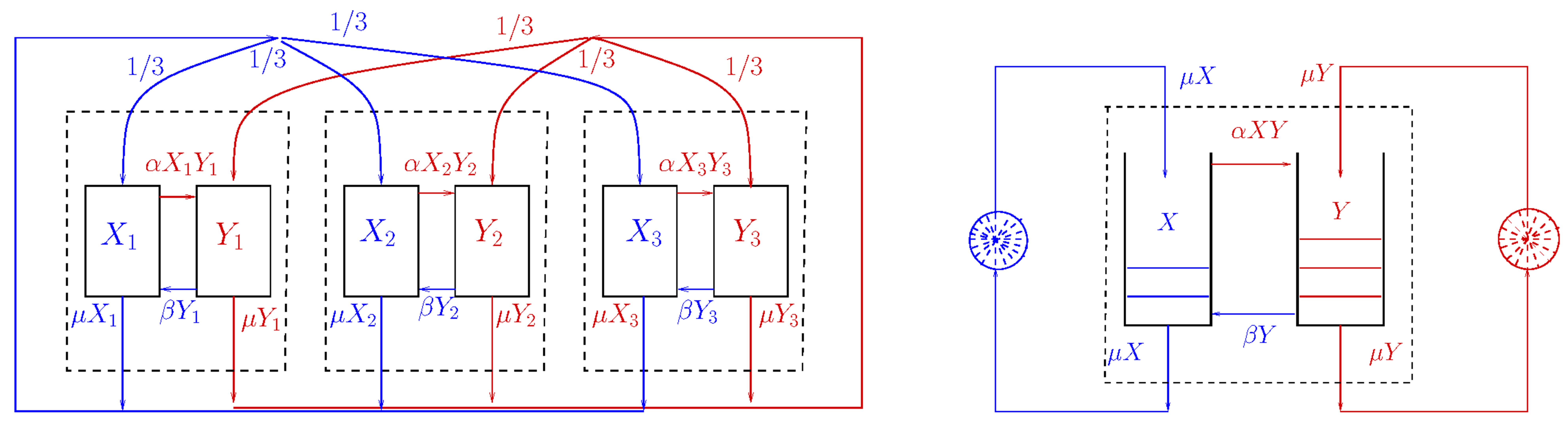}
\caption{ On the left, a closed network of 3 SIS reactors.
On the right, the SIS thermodynamic limit.}
\label{fig:figthsis}
\end{figure}

They are all infinite-dimensional Markov systems.
They can also be seen as certain non-homogeneous Markov processes (see Subsection \ref{ss:nhmc}).
As such, a wide range of asymptotic behaviors are possible for any given initial condition. 
For instance the dynamics could, e.g., converge to a stationary measure, 
or be periodic, or admit an attractor.

In this paper, for all considered cases, we assume that the thermodynamic limit
exists and satisfies the following properties on any compact of time:
\begin{itemize}
\item Stations have independent dynamics;
\item In each station, the arrival point process of susceptible (resp. infected)
is (possibly non-homogeneous) Poisson, with these two processes being independent.
\end{itemize}
This set of properties will be referred to as the thermodynamic propagation of chaos ansatz.

\subsubsection{Instances}
In all instances below, the prelimit is a closed system of $N$ stations.
Each station is again a $\cdot$/M/$\infty$ queue with departure rate $\mu$.
There is a total of $K_N$ customers 
and we assume that $\lim_{N\to \infty} K_N/N= \lambda/\mu:=\eta$.
Here $\eta$ and $\lambda$ are a positive constants
representing the mean population per station and the arrival rate in a station, respectively. 
The routing is independent and uniform at random to all stations.
\paragraph{Plain SIS}
We will call {\em SIS thermodynamic limit (SIS TL)} the infinite-dimensional system obtained
when letting $N$ to infinity, assuming $K_N$ behaves as described above.
This limit is illustrated on the right part of Figure \ref{fig:figthsis}.

\paragraph{AIR}
The {\em AIR thermodynamic limit (AIR TL)} is best described as the 
following variant of the closed system of the abstract depicted
on the left part of Figure \ref{fig:figthair} 
in any station, each susceptible customer swaps to infected with the instantaneous rate
$\alpha \frac 1 N \sum_{k=1}^N \widetilde Y_k$ if there
are $\widetilde Y_k$ infected customers in station $k$, whereas each infected customer recovers and
becomes susceptible with rate $\beta$. 
When letting $N\to \infty$, one gets a variant of the SIS thermodynamic limit (again assumed
to hold here), where stations are independent. 
In this limit, which is depicted in the right part of Figure \ref{fig:figthair},
each station behaves as a AIR AMF reactor (with
$y(t)=\E[\widetilde Y(t)]$) with the following constraint on the parameters:
the external arrival rate of susceptible is equal to
$\mu \E[\widetilde X(t)]$ and that of infected to $\mu \E[\widetilde Y(t)]$.

\begin{figure}[!h]
\centering
\includegraphics[width=0.99\linewidth]{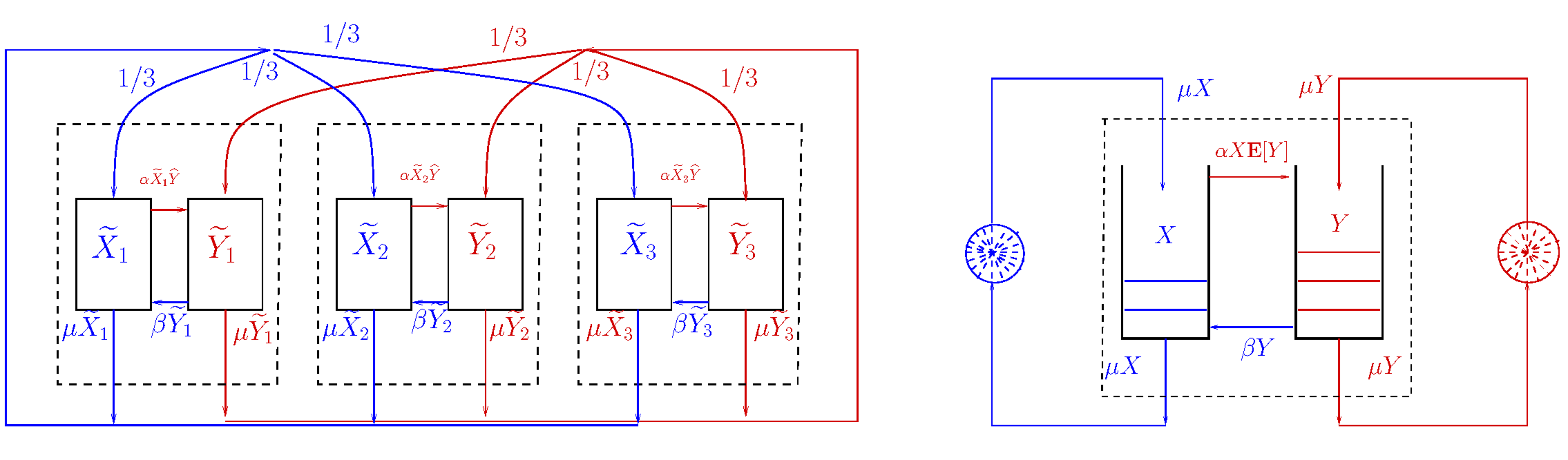}
\caption{ On the left, a closed network of 3 AIR reactors.
Here $\widehat Y = \frac 1 3(\widetilde Y_1+\widetilde Y_2+\widetilde Y_3)$.
On the right, the AIR thermodynamic limit.}
\label{fig:figthair}
\end{figure}

\paragraph{DOCS}
The {\em DOCS thermodynamic limit (DOCS TL)} features
a closed queuing network with $N$ stations and a total of $K_N$ customers 
as illustrated on the left part of Figure \ref{fig:figthdocs}.
If station $n$ has $X_n$ susceptible and $Y_n$ infected customers,
each susceptible customer swaps to infected with the instantaneous rate $\alpha Y_n$;
upon infection, it simultaneously leaves this station and is routed to one of the $N$ stations 
chosen at random. Similarly, each infected customer becomes susceptible with rate $\beta$;
upon recovery, it simultaneously leaves and is routed to a station chosen at random.
In addition, as in the closed SIS network model,
each customer (infected or susceptible) also leaves the station with a departure 
rate $\mu$ and is then also routed to a station chosen at random.
We let $N$ tend to infinity and assume that $K_N$ is such that 
the total input rate to a station tends to $\lambda$ (or equivalently that
the density of customers is $\eta$).
Then, in the limit when it exists, each station behaves as a Routing DOCS reactor
as defined in Subsection \ref{ssec:siremeli},
with the additional consistency property that, in the latter,
the external infected arrival rate $\lambda p^*(t)$ should
match the external infected departure rate $\mu \E[\widehat Y(t)]$,
and similarly, $\lambda q^*(t)$ should be equal to $\mu \E[\widehat X(t)]$.
This fixed point, which characterizes the stationary distributions of the
DOCS thermodynamic limit, 
is illustrated on the right part of Figure \ref{fig:figthdocs}.

\begin{figure}[!h]
\centering
\includegraphics[width=0.99\linewidth]{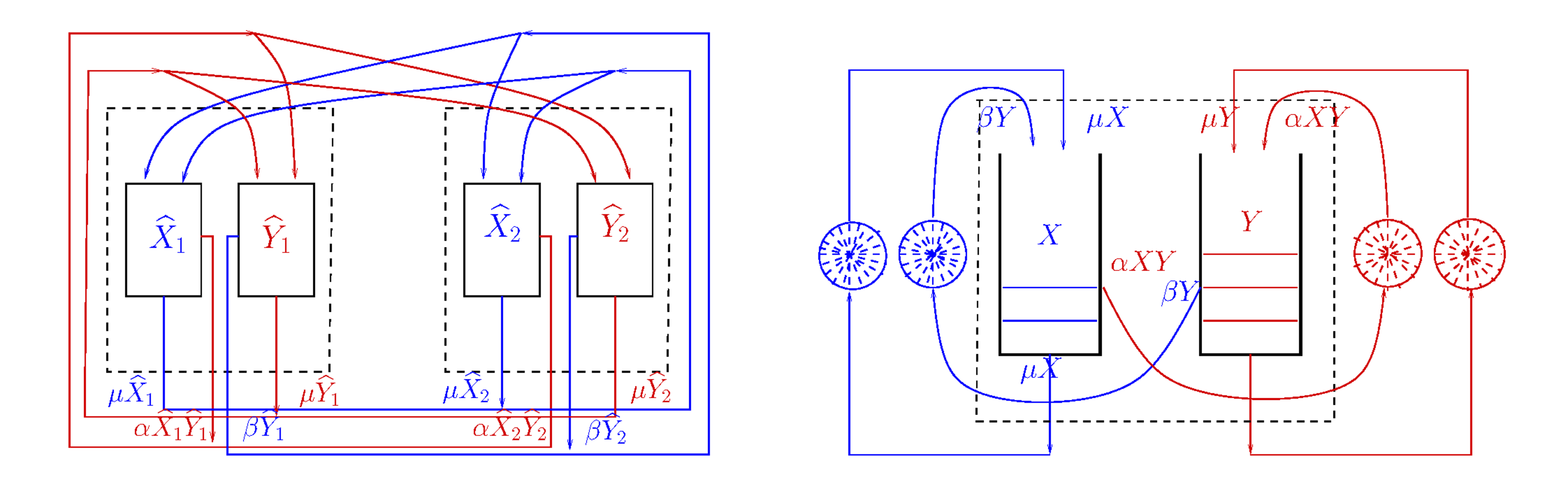}
\caption{On the left, the closed DOCS network with two stations.
At a fork, routing decisions are up or down with probability 1/2.
On the right, the DOCS thermodynamic limit.}
\label{fig:figthdocs}
\end{figure}

\subsubsection{Survival versus Extinction}
Below, for all thermodynamic limits, we assume that propagation of chaos holds.
We will show that under this assumption, the dynamic
of the infinite-dimensional Markov systems in question can be reduced to that of
a finite dimensional non-homogeneous Markov process.

\begin{Def}
\label{def:se}
In any thermodynamic limit, we will say that there is {\em survival} if the associated Markov 
system has a steady state distribution with a fraction $0<p<1$ of susceptible customers.
We will say that there is {\em weak extinction} if there exists no such $p$.
We will say that there is {\em strong extinction} if, for all initial conditions,
the associated Markov system converges to a regime without infected customers.
\end{Def}

\begin{Rem}
Weak extinction does not mean that the epidemic
vanishes in the infinite-dimensional system as time tends to infinity.
It only means that there is no initial condition that stationarizes the
infinite dimensional system in a regime with a positive fraction of infected customers.
\end{Rem}

\begin{Rem}
Like in many other models, survival in the thermodynamic limit should translate into the fact that the 
time to extinction in the prelimit at $N<\infty$ grows exponentially in $N$, where $N$ is
as defined in the abstract.
\end{Rem}

\section{SIS Reactor Analysis}
\label{sec2}

\subsection{Preliminary Observations}
It is easy to check from the PDE (\ref{eq:sispde}) that
$$\Phi(z,z)= e^{-\frac \lambda \mu (1-z)}.$$
A direct probabilistic argument also gives that
$X+Y\sim Pois (\lambda/\mu)$.

We could not solve this PDE. This is why we resort to other techniques 
based on rate conservation.
A variety of conservation equations will be discussed. These are of independent
interest and will also be used in what follows.

\subsection{First Order Relation}
\label{ss:for}
The simplest first order relation states that the rate of arrivals, $\lambda$
should match the rate of departures, namely $\mu \E[X+Y]$.
Note that there is no reason for the rate of arrivals of infected customers, $\lambda p$,
to match that of departures of infected customers, namely $\mu \E[Y]$.
This also holds for susceptible customers.
\begin{Lem}
\label{lem-for}
The following relation holds:

\begin{equation}\label{eq2}
\lambda p +  \alpha \E[XY] = (\mu +\beta) \E[Y] .\end{equation}
\end{Lem}
\begin{proof}
We get this relation by differentiating the PDE (\ref{eq:sispde}) w.r.t. $y$ and taking $x=y=1$.
\end{proof}
This receives a simple rate conservation principle interpretation.
We recognize in the LHS of (\ref{eq2}) the increase rate of the mean number of infected customers in steady state
(due to arrivals and infections), whereas the RHS is the mean decrease rate of the same quantity
due to recoveries and departures from the reactor.

Also note that by differentiating the PDE (\ref{eq:sispde}) w.r.t. $x$ and taking $x=y=1$, we get
\begin{equation}
\lambda q  +\beta \E[Y] = \mu \E[X]  +\alpha \E[XY],\end{equation}
This equation is actually the same as that in the lemma (when using the fact that $\E[X]+\E[Y]=\frac \lambda \mu$).

\subsection{Second Order Relations}

\begin{Lem} \label{lemsec}
The following relations hold:
\begin{equation}
\label{eqsecy}
(\lambda p + \mu + \beta ) \E [Y] + \alpha \E[ XY^2] =(\mu +\beta) \E[Y^2].
\end{equation}
and
\begin{equation}
\label{eqsecx}
(\lambda q + \mu) \E [X] +(\alpha+\beta) \E[XY] =\alpha \E[ X^2Y] +\mu \E[X^2].
\end{equation}
\end{Lem}

\begin{proof}
The first relation is obtained when differentiating the PDE (\ref{eq:sispde}) w.r.t. $y$ twice and taking $x=y=1$,
the second when differentiating the PDE w.r.t. $x$ twice and taking $x=y=1$.
\end{proof}

\subsection{Higher Order Relations}
\begin{Lem}
For all non-negative integers $m$ and $n$,
the RCP relation for $X^mY^n$ reads
\begin{eqnarray}
& & \hspace{-1cm}\lambda p \E[X^m\left((Y+1)^n-Y^n\right)] +\lambda q \E[\left((X+1)^m-X^m\right)Y^n]\nonumber \\
& + &
\beta \E[Y\left((X+1)^m(Y-1)^n-X^mY^n\right)]\nonumber \\
& +&
\alpha \E[XY\left((X-1)^m(Y+1)^n-X^mY^n\right)]\nonumber \\
& + &
\mu \E[X\left((X-1)^m-X^m\right)Y^n ]\nonumber\\
& + & \mu \E[YX^m\left((Y-1)^n-Y^n\right)]=0.
\end{eqnarray}
\end{Lem}

The proof is based on Equations (\ref{eqsecy-rcp}), (\ref{eqsecx-rcp}) and
(\ref{eqsecxy-rcp}) in Appendix \ref{app-sec-crc}.

\section{DOCS Reactor Analysis} 
\label{sec:SIDR}

\subsection{Preliminary Observations}
\label{ssec:SIDR1}
We start with a few observations on this system:
\begin{itemize}
\item By either a simple probabilistic argument or by a direct analytical derivation based on the PDE,
we get that $\Phi(1,y)=e^{-\theta(1-y)}$, that is $Y$ is Poisson with mean
$\theta=\frac{\lambda p}{\nu}$ (the $Y$ process is that of an autonomous $M/M/\infty$ queue
with input rate $\lambda$ and service rate $\nu$);
\item Here, the random variable $X+Y$ has no reason to be Poisson;
\item The stationary output rate of infected, $\nu \E[Y]$, should match the input rate
of infected ($\lambda p$). This is because infected customers evolve as those of an M/M/$\infty$ queue;
\item The stationary depletion rate of susceptible, $\mu \E[X]+ \alpha \E[XY]$, 
should match the increase rate of susceptible ($\lambda q$).
This is because susceptible evolve as in a stationary queue.
\end{itemize}
We summarize the last two observations in the following conservation equations:
\begin{eqnarray}
\lambda p  & = & \nu\E[Y]\\
\lambda q  & = & \mu\E[X] +\alpha \E[XY].
\end{eqnarray}

\subsection{Analytical Solution}
\label{sec:sid-analy}

By differentiating Equation \eqref{eq:sipde} w.r.t. $x$ and taking $x=1$, we get
\begin{eqnarray*} 
-\lambda q \Phi(1,y) 
+\left(\lambda p (1-y) +\mu\right) \Phi_x(1,y) 
= \left(\nu (1-y) -\alpha y\right) \Phi_{xy}(1,y),
\end{eqnarray*}
or equivalently, the function $\Psi(y):= \Phi_x(1,y)$ satisfies the ODE
\begin{eqnarray} 
\label{eq:ode1}
\left(\nu (1-y) -\alpha y\right) \Psi'(y)=
\left(\lambda p (1-y) +\mu\right) \Psi(y) -\lambda q e^{-\theta(1-y)}. 
\end{eqnarray}
This can be rewritten as the first order ODE
\begin{eqnarray} 
\Psi'(y)= h(y) \Psi(y) - g(y),
\end{eqnarray}
with 
\begin{equation}h(y):= 
\frac {\lambda p (1-y) +\mu} {\nu (1-y) -\alpha y}, \qquad
g(y):=
\frac{ \lambda q e^{-\theta(1-y)}} {\nu (1-y) -\alpha y}  .
\end{equation}
Notice that the coefficients of this first order ODE have a singularity at
$$y=y^*:= \frac{\nu}{\nu+\alpha}<1.$$ 
One can nevertheless work out the solution of this ODE for y in a neighborhood of 1 not including $y^*$.
The homogeneous equation has for solution
$$ T(y):= e^{\int_1^y h(z) dz}.$$
Hence when looking for a solution of the form
$\Psi(y)=T(y) a(y)$, we get that
$$ a(y)= a(1) - \int_1^y   g(u) e^{-\int_1^u h(z) dz} du.$$ 
Hence
\begin{eqnarray} 
\Psi(y)= e^{\int_1^y h(z) dz}\left(
\Psi(1) - \int_1^y   g(u) e^{-\int_1^u h(z) dz} du\right).
\end{eqnarray}
It is easy to check that for $u$ in a right neighborhood of $y^*$,
$e^{-\int_1^u h(z) dz}\sim K (u-y^*)^{b}$ with $K$ a constant and
$$b=\frac{ \lambda p \alpha +\mu(\alpha+\nu)}{(\alpha +\nu)^2}.$$
Hence the integral
$$\int_1^{y^*} g(u) e^{-\int_1^u h(z) dz} du$$
is well defined and finite despite the singularity of $g$ at $y^*$.
This integral must match $\Psi(1)$. Indeed, if it were not the
case, the function $\Psi$ would have a singularity at $y^*$, which is not possible.
Hence
\begin{equation*}\Psi(1)=\int_1^{y^*} g(u) e^{-\int_1^u h(z) dz} du,\end{equation*}
that is
\begin{equation}\E[X]=\int_1^{\frac{\nu }{\nu+\alpha}}
\frac{ \lambda q e^{-\frac{\lambda p}\nu (1-u)}} {\nu (1-u) -\alpha u} 
\exp\left(-\int_1^u 
\frac {\lambda p (1-z) +\mu} {\nu (1-z) -\alpha z} dz\right) du.
\end{equation}

Using the fact that
$$ 
\int_1^u 
\frac {\lambda p (1-z) +\mu} {\nu (1-z) -\alpha z} dz
=-\frac 1 {\nu +\alpha} \left(
\lambda p (1-u) + \left(\mu +\lambda p\frac{\alpha}{\alpha +\nu}\right) \ln \left(\frac{u-y^*}{1-y^*}\right) \right),
$$
with $y^*=\frac \nu{\nu +\alpha}$,
we get that
\begin{eqnarray*}
\E[X] & = &
\lambda q 
\int_1^{\frac{\nu }{\nu+\alpha}}
e^{-\frac{\lambda p}\nu (1-u)}
e^{\frac{\lambda p} {\nu+\alpha} (1-u)}
\frac 1 {\nu (1-u) -\alpha u} 
\left( \frac{u-y^*}{1-y^*} \right)^{\frac \mu {\nu +\alpha} +\lambda p\frac{\alpha}{(\nu+\alpha)^2}}
du .
\end{eqnarray*}
That is
\begin{eqnarray}
\E[X] & =  & \frac {\lambda q}
{
(\nu+\alpha)
\left(\frac{\alpha}{\nu+\alpha}\right)^{\frac \mu {\nu +\alpha} +\lambda p\frac{\alpha}{(\nu+\alpha)^2}}
}
\int_{\frac{\nu }{\nu+\alpha}}^1
e^{-\frac{\lambda p \alpha}{\nu(\nu +\alpha)} (1-u)}
\left( u-\frac \nu {\nu+\alpha} \right)^{\frac \mu {\nu +\alpha} +\lambda p\frac{\alpha}{(\nu+\alpha)^2}-1}
du .
\end{eqnarray}

Applying the change of variables
\begin{align*}
u= t \frac{\alpha}{\nu +\alpha} + \frac{\nu}{\nu +\alpha},
\end{align*}
we get
\begin{eqnarray}
\label{eq:odesidpure}
\E[X] & =  & \frac {\lambda q}
{\nu+\alpha
}
\int_{0}^1
e^{-\frac{\lambda p \alpha^2}{\nu(\nu +\alpha)^2} (1-t)}
t^{\frac \mu {\nu +\alpha} +\lambda p\frac{\alpha}{(\nu+\alpha)^2}-1}
dt .
\end{eqnarray}

\subsection{Analysis of the Routing DOCS Reactor}

Compared to the SIS Reactor, the Routing DOCS reactor (defined in Subsection \ref{ssec:siremeli})
features much faster migration. This is because the
natural departures (those happening with rate $\mu$ regardless of the state)
are complemented by departures that are caused by a change of state.
Due to the nature of the mean-field model, this goes with an increased
arrival rate of both infected and susceptible customers, in that `external' arrivals
(those happening with rate $\lambda p(t)$ for infected and rate $\lambda q(t)$ for
susceptible) are complemented by 'internal' arrivals (with respective rates
$\alpha \E[\widehat X(t)\widehat Y(t)]$ for infected, and
$\beta \E[\widehat Y(t)]$ for susceptible).
As we shall see (Lemma \ref{lem:stilhold} below),
this accelerated (in and out) migration preserves the mean queue size
in steady state when it exists.

\subsubsection{Stationary Regime of the Routing DOCS Reactor}

An important question is whether 
there exists a stationary regime for this reactor. A necessary condition is that
there exists a probabilistic solution to the PDE (\ref{eq:sipde})
satisfying the above consistency equations.
If it is the case, it follows from the observations preceding the remark of
Subsection \ref{ssec:siremeli} and the results of Section \ref{ssec:SIDR1} that
$\widehat Y$ is Poisson with parameter
\begin{equation}\widehat \theta= \frac {\lambda p +\alpha \E[\widehat X \widehat Y]} {\mu + \beta}.
\end{equation}
The random variable $\widehat X+\widehat Y$ has no reason to be Poisson with the two terms
of the sum independent.

\subsubsection{First Order Rate Conservation}
Assuming that the routing DOCS reactor has a stationary regime,
one can apply the rate conservation equations to this system.
The first order RCP, which says that the input and output rates coincide, reads
\begin{equation}
\label{eq:rmf1x}
\lambda q +\beta \E[\widehat Y] = \mu \E[\widehat X] +\alpha \E[\widehat X \widehat Y]
\end{equation}
for susceptible customers and
\begin{equation}
\label{eq:rmf1y}
\lambda p + \alpha \E[\widehat X \widehat Y] = (\beta + \mu) \E[\widehat Y] 
\end{equation}
for infected ones. Note that these
match the first order relations of the SIS reactor.
\begin{Lem}
\label{lem:stilhold}
In the stationary regime of the DOCS routing mean-field model, one also has
\begin{equation}
\label{eq:stilhold}
\E[\widehat Y] +\E[\widehat X] = \frac {\lambda} {\mu}.
\end{equation}
\end{Lem}
\begin{proof}
The result is obtained by adding up the last two equations.
\end{proof}
The fact that the mean total population is the same as in the
SIS model is remarkable as, here, and in contrast to the SIS reactor,
$\widehat X +\widehat Y$ is not Poisson.

\subsubsection{Higher Order Rate Conservation}
The setting is that of the last subsection.
The RCP for $\widehat Y(\widehat Y-1)$ reads
\begin{equation}
\label{eq:rmf2y}
\left(\lambda p + \mu +\beta\right) \E[\widehat Y]
+ \alpha \E[\widehat X \widehat Y] \E[\widehat Y]
= (\beta + \mu) \E[\widehat Y^2] .
\end{equation}
Note that it is similar to 
(\ref{eqsecy}), except that $\E[XY^2]$ is replaced by $\E[\widehat X\widehat Y] \E[\widehat Y]$.
Similarly, the RCP for $\widehat X(\widehat X-1)$ reads
\begin{equation}
\label{eq:rmf2x}
(\lambda q + \mu) \E[\widehat X] + \beta \E[\widehat Y] \E[\widehat X]
+\alpha \E[\widehat X \widehat Y] 
= \alpha \E[\widehat X^2\widehat Y] + \mu \E[\widehat X^2], 
\end{equation}
which is similar to 
(\ref{eqsecx}), except that $\beta \E[XY]$ is replaced by  $\beta \E[\widehat X] \E[\widehat Y]$.
Finally, the RCP for $\widehat X \widehat Y$ reads
\begin{equation}
\label{eq:rmf2xy}
\left(\lambda q + \beta \E[\widehat Y]\right) \E[\widehat Y]
+\left(\lambda p + \alpha \E[\widehat X \widehat Y]\right) \E[\widehat X]
= (\beta + 2 \mu) \E[\widehat X \widehat Y] 
+ \alpha \E[\widehat X \widehat Y^2].
\end{equation}
Note that the complexity of the last two equations remains similar to that
of the corresponding equations in the initial system.

\section{AIR Reactor Analysis}
\label{sec:airR}

\subsection{The Reactor}

Thanks to the linearity of the rates, this AIR queuing network defined in Section \ref{sec:sum}
has a product-form distribution in steady state which is the product of two
Poisson distribution, with parameter $\frac {\lambda_1}{\mu}$ in station 1
and $\frac {\lambda_2}{\mu}$ in station 2.
Direct computations based on the traffic equations
\begin{eqnarray*}
\lambda_1 & = &  \lambda q +\lambda_2 \frac \beta {\mu +\beta}\\
\lambda_2 & = &  \lambda p +\lambda_1 \frac {\alpha y} {\mu +\alpha y}
\end{eqnarray*}
give that
$$\lambda_1 = \frac{(\mu+\alpha y)(\beta +\mu q)\lambda}{(\mu +\beta)(\mu+\alpha y)-\beta\alpha y}\ .$$
We also have $x=\frac {\lambda_1}{\mu+\alpha y}=\frac \lambda \mu -y$, with $x$ the mean queue
size in station 1.

\subsection{The Averaging Mean Field Case}
Consider the averaging mean-field limit when it exists.
Denote by $\widetilde X$ the stationary number of susceptible customers
in the typical station and by $\widetilde Y$ the stationary number of infected ones in this system.
By arguments similar to those in Subsection \ref{ss:for}, these state variables satisfy the relations
\begin{equation}
\label{eq6bis}
\lambda q
+ \beta \E[\widetilde Y] 
= \mu \E[\widetilde X] 
+ \alpha \E[\widetilde X]\E[\widetilde Y],\end{equation}
where the independence comes from the fact that in the limit, the infection rate is constant
and equal to $\alpha y=\alpha \E[\widetilde Y]$. Using this and the fact that
$\E[\widetilde X]=\frac \lambda \mu -y$, one gets
\begin{equation}\label{eq:traf}
\lambda q =
\mu\left(\frac \lambda \mu - y\right)-\beta y + \alpha y
\left(\frac \lambda \mu - y\right).
\end{equation}
Thus, in the averaging mean-field version of the SIS reactor,
\begin{equation}  \E[\widetilde X] = \frac{\mu +\beta+\alpha \frac \lambda \mu 
-\sqrt{(\mu +\beta+\alpha \frac \lambda \mu )^2 -4\alpha\lambda \left(q+\frac \beta \mu \right)}}{2\alpha}
\end{equation}
and
\begin{equation}
\label{eq:19}
 \E[\widetilde Y] = \frac \lambda \mu -\frac{\mu +\beta+\alpha \frac \lambda \mu 
-\sqrt{(\mu +\beta+\alpha \frac \lambda \mu )^2 -4\alpha\lambda \left(q+\frac \beta \mu \right)}}{2\alpha}.
\end{equation}

\section{SIS Thermodynamic Limit Analysis}
\label{sec:thermo}

The natural parameters of the SIS thermodynamic limit are
$(\eta,\mu,\alpha,\beta)$ with $\eta$ the density parameter 
(that is the mean number of customers per station, whatever their state), 
$\mu$ the motion rate, $\alpha$ the infection rate, and $\beta$ the recovery rate.
Note that the arrival rate in a station (whatever the state) is then $\lambda=\eta \mu$. Another natural
parameterization is hence $(\lambda,\mu,\alpha,\beta)$.

\subsection{Fixed-Point Equations for the SIS Reactor}
\label{sec3}
If this thermodynamic model admits a stationary regime, then 
there exists, in the open loop SIS reactor with parameters $\eta,\mu,\alpha,\beta$,
a value of $p$, say $p^*$, such that, for this value of $p$,
the external arrival rate $\lambda p$ of infected matches the 
external departure rate of infected, namely $\lambda p^*=\mu \E[Y]$
(or equivalently $g(p^*)=p^*$).
Similarly, we should also have $\lambda q^*= \mu \E[X]$.

Note that since the total arrival rate of infected, $\lambda p^* +\alpha \E[XY]$,
necessarily matches the 
total departure rate of infected, $(\mu+\beta) \E[Y]$, the fact that
$\lambda p^*=\mu \E[Y]$ is equivalent to the fact that $\alpha \E[XY] = \beta \E[Y]$.

\subsubsection{Arrival versus Departure Rate of Infected Customers}

Let $p_o$ denote the proportion of infected customers in the departure process of
the SIS reactor in the steady state, namely
$$ p_o= \E[Y] \frac \mu \lambda.$$
Here is another representation of $p_0$ obtained when
considering departures within the first busy cycle of an $M/M/\infty$ queue
(i.e., the cycle starting when the queue moves from empty to busy and ending at the next
event of this type). By the SLLN, $p_o$ may be represented as
\begin{align*}
p_o =  \frac{{\mathbf E} D_I(T)}{{\mathbf E} D(T)},
\end{align*}
where $T$ is the duration of a typical busy cycle, $D(T)$ the number of customers served within the cycle,
and $D_I(T)$ the number of customers departed from the system in the infected state within the cycle. 

Let us fix strictly positive parameters $\mu,
\alpha$ and $\beta$ and consider $p_o$ as a function of $p$ and $\eta$ only, say 
\begin{align}
\label{eq:gfunk}
p_o=g(p,\eta).
\end{align} 
If $\eta$ is fixed too, then we write $g(p)$ instead of $g(p,\eta)$.
Clearly, $g(0)=0$ and $g(1)<1$.

The next lemmas are proved in Appendix \ref{secap1}.

\begin{Lem}\label{lem:pop}
Function $g$ is an
increasing, strictly concave, and differentiable function on $[0,1]$.
\end{Lem}

\begin{Lem}\label{lem:unic}
Depending on parameters $\eta,\mu,\alpha$ and $\beta$, either there is only one solutions $p=0$ to 
the equation 
\begin{align}\label{fixedpoint}
p=g(p)
\end{align} 
or there are exactly two solutions, $p=0$ and 
\begin{align}\label{plambda}
p^*\equiv p^*(\eta,\mu,\alpha,\beta)\in (0,1).
\end{align}
Let $g'(0)$ denote the right derivative of the function $g(p)$ at 0,
which exists due to the previous lemma. 
Then 
\begin{align}\label{p*exists}
p^* \ \ \mbox{exists if and only if}  \ \ g'(0)>1.
\end{align}
\end{Lem}

\subsubsection{Main Results on Fixed Point}

In this subsection,
we consider $p_o$ as a function of the two parameters, $p$ and $\eta$,
$p_o= g(p,\eta)$.
The proof of the following lemma is given in Subsection \ref{secap1}.

\begin{Lem}\label{lem8}
The function 
$$\eta \to g'(0,\eta)$$
is strictly increasing.
\end{Lem}

\begin{Cor}\label{Cor10}
There exists a function $\eta_c^{(s)}(\alpha,\beta,\mu)=\eta_c(\alpha,\beta,\mu) \in [0,\infty)\cup \{\infty\}$ such that
\begin{itemize}
\item{} $g'(0,\eta)\le 1$ $\Leftrightarrow$
$p=0$ is the only solution to \eqref{fixedpoint} $\Leftrightarrow$ $\eta \le \eta_c$;
\item{} $g'(0,\eta)>1$ $\Leftrightarrow$
there are two solutions to Equation \eqref{fixedpoint}, $p=0$ and $p^*>0$, $\Leftrightarrow$
$\eta >\eta_c$.
\end{itemize}
\end{Cor} 

\begin{proof}
These results follow directly from the previous lemmas.
\end{proof}

\subsection{Bounds}
\label{sec28}

In this subsection, we provide sufficient conditions for the existence and
the non-existence of a solution $p^*>0$ to
the equation $p_o(p)=p$. We also show that $\eta_c^{(s)}$ is finite and strictly positive.

\begin{Lem}\label{Lem11} \label{lem:out_bigger_in_small_p_large_lambda}
Fix the set of parameters $\alpha$, $\beta$, and $\mu$. For $\eta$ (or equivalently $\lambda$) large enough,
there exists $p_1 > 0$ such that if $0<p < p_1$, then the fraction of infected customers
of the output is larger than that of the input, namely $p_{o} > p$.
\end{Lem}

\begin{proof} Note that an infected individual is guaranteed to leave infected if it
departs before it recovers. A susceptible individual (let us refer to it as target individual in what follows
in this paragraph) is guaranteed to leave as infected if an infected individual (we will
refer to it is such in what follows in this paragraph) arrives before the susceptible one departs,
then the infected individual infects the target one before the departure of either of them,
and the recovery of the infected one, and then: either the target individual departs before its
own recovery, or it recovers before the departure of either one and recovery of the infected
one but gets infected by the infected one again, and so on. Thus,
\begin{align}\label{aaa}
p_{o} \ge p \frac{\mu}{\mu+\beta} + (1-p) \frac{\lambda p}{\lambda p + \mu} \kappa,
\end{align}
where

\begin{align}\label{kappa}
\kappa = \frac{\alpha}{\alpha+2 \mu + \beta} \left(\frac{\mu}{\mu+\beta} + \frac{\beta}{2(\mu+\beta)}a\right)
 = \frac{2 \alpha \mu}{2(\mu+\beta)(\alpha+2\mu + \beta) - \alpha \beta}.
\end{align}
Note that $p_o>p$ follows if the RHS of \eqref{aaa} is strictly bigger than $p$, which is equivalent to
\begin{align*}
\frac{\mu}{\mu+\beta} + (1-p) \cdot \frac{\lambda \kappa}{\lambda p+\mu} >1,
\end{align*}
or
\begin{align*}
(1-p)\frac{\lambda \kappa}{\lambda p+\mu} > \frac{\beta}{\mu+\beta}.
\end{align*}

Assume that, for some $C>1$, 
\begin{equation} \label{eq:large_lambda_estimate}
\lambda \kappa > \frac{C\mu \beta}{\mu+\beta},
\end{equation}
then $p_o > p$ for
$$
p < p_1 \equiv \frac{C-1}{\lambda/\mu +C},
$$
so that there exists a $p^*>0$ solving $p_o(p)=p$. 
\end{proof}

\begin{Cor}\label{Cor100}
It follows from Lemma \ref{Lem11} that for all $\alpha,\beta,\mu$,
the value of $\eta_c^{(s)}$ in Corollary \ref{Cor10} is finite, and more precisely
\begin{equation}
\label{eq:lub}
\eta_c^{(s)} \le \frac 1 \kappa \frac{\beta}{\mu+\beta},
\end{equation}
with $\kappa$ defined in (\ref{kappa}).
\end{Cor}

Here is another observation:

\begin{Lem} \label{lem:out_smaller_in_any_p_small_lambda}
Fix the set of parameters $\alpha$, $\beta$, and $\mu$. For $\eta$ (or equivalently $\lambda$) small enough,
$p_{o} < p$ for all $p$.
\end{Lem}

\begin{proof} An individual is guaranteed to leave the system as susceptible in one of two scenarios. In the first scenario, it arrives as susceptible, finds no infected individuals in the system upon its arrival and leaves before the arrival of an infected individual. The probability of this is
$$
(1-p) \P(Y=0) \frac{\mu}{\mu+\lambda p}.
$$
The above can be bounded from below by
\begin{align*}
& (1-p)(1-\E(Y)) \left(1-\frac{\lambda p}{\mu+\lambda p}\right) = (1-p)\left(1-p_{o} \frac{\lambda}{\mu}\right)\left(1-p\frac{\lambda}{\mu}\right)
\\ & \ge \left(1-p\left(1+\frac{\lambda}{\mu}\right)\right)\left(1-p_{o} \frac{\lambda}{\mu}\right) \ge 1-p\left(1+\frac{\lambda}{\mu}\right) -p_{o} \frac{\lambda}{\mu}.
\end{align*}
Finally, for a fixed $\mu$ and a small $\varepsilon>0$, one can take $\lambda \le \varepsilon \mu$ so that the above is bounded from below by
$$
1-p(1+\varepsilon) - \varepsilon p_{o}.
$$
In the second scenario, an individual arrives infected, finds no other individual in the system, recovers before any other individual arrives and then leaves before any infected individual arrives. The probability of this is
$$
p e^{-\lambda/\mu} \frac{\beta}{\beta+\lambda+\mu}\frac{\mu}{\mu+\lambda p} \ge p \left(1-\frac{\lambda}{\mu}\right)\frac{\beta}{\beta+\lambda+\mu}\left(1-\frac{\lambda p}{\mu}\right) \ge p \left(1-\frac{\lambda}{\mu}\right)^2\frac{\beta}{\beta+\lambda+\mu}.
$$
With the choice of small $\lambda$ already made, the above is bounded from below by
$$
p(1-\varepsilon)^2 \frac{\beta}{\beta+\mu+\varepsilon\mu} \ge p(1-\varepsilon)^2 \frac{\beta}{\beta+\mu} \left(1-\varepsilon \frac{\mu}{\beta+\mu}\right) \ge p c (1-3 \varepsilon),
$$
where $c = \beta/(\beta+\mu)$.
Combining the two scenarios, we conclude
$$
1-p_{o} \ge 1-p(1+\varepsilon) - \varepsilon p_{o} + p c (1-3\varepsilon),
$$
or equivalently
$$
p_{o} \le  p \frac{1-c+3 c\varepsilon}{1-\varepsilon} < p
$$
as long as $\varepsilon < c/(2+3c)$.
\end{proof}

From the proof of Lemma \ref{lem:out_smaller_in_any_p_small_lambda}, one gets the following bound. 
\begin{Cor}\label{Cor101}
For all $\alpha,\beta,\mu$,
the value of $\eta_c^{(s)}$ in Corollary \ref{Cor10} is strictly positive, and more precisely
\begin{equation}
\label{eq:lubi2}
\eta_c^{(s)} \ge \frac{\beta}{2\mu+5\beta}.
\end{equation}
\end{Cor}

\begin{Rem}
Note that the last lower bound does not depend on $\alpha$.  
This may look surprising at first glance.
The fact that this bound holds even for $\alpha =\infty$ can be explained as follows.
For any fixed $\eta >0$, in the $M/M\infty$ queue, there are busy cycles with only one customer.
The smaller is $\eta$, the closer to 1 is the probability $\mu/(\lambda+\mu)$ for a typical
busy cycle to have only one customer to be served.
Call such a cycle a 1-cycle. If a customer enters the queue as
infected and if it is served in a 1-cycle, it has a chance
close to $\beta/(\beta+\mu)$ to recover before leaving the queue; and if it susceptible,
it leaves the queue susceptible after service in a 1-cycle queue.
So $p_0>p$ uniformly in all $\alpha$'s, for all $\eta$ sufficiently small.
\end{Rem}

\subsection{Survival vs. Extinction}
The limiting system has 3 parameters only. Indeed, one can always take, say, $\beta=1$ by a 
time rescaling and only the three other parameters remain.
Consider the parameterization $(\alpha,\mu,\eta)$ and the associated positive orthant.

\subsubsection{Phase diagram in $\eta$}
The two following results rely on the definitions of extinction
and survival given in Definition \ref{def:se}.
They show that the SIS thermodynamic limit admits simple phase diagram
w.r.t. $\eta$, with critical value equal to $\eta_c^{(s)}$ defined above.
\begin{Thm}
\label{thm:surv}
If the SIS thermodynamic propagation of chaos ansatz holds, then,
in the SIS thermodynamic limit, there exists a constant $\eta_c$
such that there is survival if $\eta > \eta_c^{(s)}$.
\end{Thm}

\begin{proof}
From Corollary \ref{Cor10}, if $\eta > \eta_c^{(s)}$, there exists $p^*>0$ such that
$g(p^*)=p^*$. If the ansatz holds, take as initial distribution in the non-homogeneous
Markov representation of the SIS thermodynamic limit the law of the open SIS reactor with
infected input rate $p=p^*$. This distribution is a stationary distribution of this system.
\end{proof}

\begin{Thm}
\label{thm:ext}
If the SIS thermodynamic propagation of chaos ansatz holds, then,
in the SIS thermodynamic limit, there is strong extinction if $\eta \le \eta_c^{(s)}$.
\end{Thm}

\begin{proof}
The proof is given in Subsection \ref{ss:pttext}.
\end{proof}

\begin{Rem}
Theorem \ref{thm:surv} only shows that there exists an initial distribution ${\cal P}^*$ of the SIS
thermodynamic limit such that makes this limiting system stationary.
A stronger result is proved in Corollary \ref{Squeeze} in the appendix:
for all initial distributions which are stochastically larger than ${\cal P}^*$,
the distribution of the SIS thermodynamic limit
converges to ${\cal P}^*$ as time tends to infinity.
\end{Rem}
\subsubsection{Phase diagrams in $\alpha$ and in $\beta$}

Consider the SIS reactor.
Fix all parameters except $\alpha$ (resp. $\beta$) and consider
$0<\alpha_1<\alpha_2$
(resp. $0<\beta_2<\beta_1$). It is shown in Appendix \ref{ss:mab}
that there exists a coupling of
the two associated models such that $X_1(t)+Y_1(t)=X_2(t)+Y_2(t)=:N(t)$ and
$Y_1(t)\le Y_2(t)$ a.s., for all $t\ge 0$, given that the two models start from the same initial condition
$X_1(0)=X_2(0)$ and $Y_1(0)=Y_2(0)$ a.s.

Consider now a finite closed network of SIS reactors with $N$ stations and $K_N$ customers.  
It is shown in Appendix \ref{ss:mab} that the same monotonicity properties hold.

By arguments similar to those used w.r.t. $\eta$, we have

\begin{Thm}
\label{thm:survab}
If the SIS thermodynamic propagation of chaos ansatz holds, then,
in the SIS thermodynamic limit, there exists a constant $\alpha_c^{(s)}$
(resp. $\beta_c^{(s)}$)
such that there is survival if $\alpha > \alpha_c^{(s)}$
(resp. $\beta < \beta_c^{(s)}$) and strong extinction otherwise.
\end{Thm}

\subsubsection{Other Phase diagrams}

Phase diagrams w.r.t. $\mu$ will only be discussed numerically in Section \ref{ss:num}.
An instance of question of interest is whether there is a simple (monotonic as above) phase diagram
w.r.t. $\mu$ when fixing $\alpha,\beta$ and $\eta$.

\subsection{First and Second Order Relations in the Thermodynamic Limit}
In the thermodynamic limit, 
the first order relation of Lemma \ref{lem-for} simplifies to
\begin{equation}\label{eq2therm}
\alpha \E[XY] = \beta \E[Y].
\end{equation}

Consider now the second order relations of Lemma \ref{lemsec}.
Equation (\ref{eqsecy}) can be rewritten as 
\begin{equation}
\label{eqsecytherm1}
(\eta\mu p + \mu + \beta ) p \eta + \alpha \E[ XY^2] =(\mu +\beta) \E[Y^2]
\end{equation}
or equivalently
\begin{equation}
\label{eqsecytherm1plus}
(\mu + \beta ) (\E[Y^2] -\E[Y]^2 -\E[Y])=
\alpha \E[XY] \E\left[Y \frac {\alpha XY}{\E[\alpha XY]}\right] -\beta\E[Y] =
\beta \E[Y](\E_I^0[Y^-]-\E[Y]).
\end{equation}
Similarly, (\ref{eqsecx}) can be rewritten as
\begin{equation}
\label{eqsecxtherm1}
(\eta\mu q + \mu) \eta q +(\alpha+\beta) \E[XY] =\alpha \E[ X^2Y] +\mu \E[X^2]
\end{equation}
or equivalently
\begin{eqnarray}
\label{eqsecxtherm1plus}
\mu \left(\E[X^2] -\E[X]^2-\E[X]\right) 
& = & (\alpha+\beta) \E[XY] -\alpha\E[XY] \E\left[X \frac {\alpha XY}{\E[\alpha XY]}\right]\nonumber
\\ & = & \E[XY] \alpha \left( 1 + \frac{\beta}{\alpha} - \E_I^0[X^-]\right)
 =  \E[XY] \alpha \left(\frac{\beta}{\alpha} - \E_I^0[X^+]\right).
\end{eqnarray}

Similarly, Equation (\ref{eqsecy-ba}) can be simplified to
\begin{equation}
\label{eqsecytherm2}
\beta (\E_I^0[Y^-]-\E_R^0[Y^+]) = \mu \E_{D_I}^0[Y^+] -\eta\mu p
\end{equation}
whereas (\ref{eqsecx-ba}) can be simplified to
\begin{equation}
\label{eqsecxtherm1bis}
\beta( \E_{R}[X^-] - \E_{I}[ X^+])= \mu \frac q p \E_{D_s}[X^+]-\eta\mu q \frac q p. 
\end{equation}

\section{AIR Thermodynamic Limit Analysis}
\label{secAIRth}
Let
\begin{equation}
\label{eq:lsa}
\eta_c^{(a)} := \frac \beta \alpha.
\end{equation}
\begin{Thm} Under the AIR TL propagation of chaos ansatz,
if $\eta \le \eta_c^{(a)}$, then there is weak extinction, whereas
if $\eta> \eta_c^{(a)}$, there is survival. 
In addition, if there is survival, in the stationary regime of the thermodynamic limit,
\begin{enumerate}
\item $\E[\widetilde X] = \frac \beta \alpha $;
\item $\E[\widetilde Y]= \eta  - \frac \beta \alpha$;
\item $ 1-\widetilde p^*= \widetilde q^* = \frac{\beta}{\eta \alpha}$;
\item $ \widetilde X$ and $\widetilde Y$ are independent and Poisson.
\item The departure rate from a station is $\lambda=\eta \mu$.
\end{enumerate}
If there is extinction, in the stationary regime,
\begin{enumerate}
\item $\E[\widetilde X] = \frac \beta \alpha $;
\item $\E[\widetilde Y]= 0$;
\item $ 1-\widetilde p^*= \widetilde q^* = 1$;
\item $ \widetilde X$ is Poisson.
\item The departure rate from a station is $\lambda=\eta \mu$.
\end{enumerate}
\end{Thm}
\begin{proof}
Assuming existence of the thermodynamic limit and its convergence to a stationary distribution,
the result is immediate when plugging in $y=p \eta$ in (\ref{eq:traf}).
\end{proof}

The AIR phase diagram is hence quite explicit. 
We have $\eta_c= \frac{\beta}{\alpha}$, which does not depend on $\mu$. In addition,
the following continuity property holds:
when $\eta\downarrow \eta_c$, $\E[\widetilde Y] \downarrow 0$.
Similar statements can be coined w.r.t. any parameter other than $\eta$.

\section{DOCS Thermodynamic Limit Analysis}
\label{secDOCSth}

Since the behavior of a station in DOCS thermodynamic limit is a specific instance of 
that of a station in the routing mean-field model, the following results hold on the former:
\begin{itemize}
\item The assumption that the 'external' arrival
rate is $\lambda$ is equivalent to the assumption that the mean number
of customers (of both types) in a station in steady state is $\frac{\lambda}{\mu}$
(see Eq. (\ref{eq:stilhold}). By symmetry, this is equivalent to assuming that
$\lim_{N\to \infty} K_N/N= \lambda/\mu:=\eta$ in the DOCS thermodynamic limit.
\item The condition $\lambda p^*= \mu \E[\widehat Y]$ is 
equivalent to $\alpha \E[\widehat X \widehat Y]= \beta \E[\widehat Y]$ in view
of (\ref{eq:rmf1y}). In words, in the DOCS thermodynamic limit, the stationary
rate of 'in station' infections matches the stationary rate of 'in station' recoveries.  
\end{itemize}

\subsection{Analytical solution}
In this section use $p$ in place of $p^*$ for the sake of light notation.
\subsubsection{Rate Conservation Equations}
As shown above, in the thermodynamic limit, we have
\begin{eqnarray}
\lambda p & = & \mu \E[\widehat Y],\label{rcpsid1}
\end{eqnarray}
in words, the rate of 'natural' migration of infected customers matches that of 'external' arrivals
of infected customers,
\begin{eqnarray}
\lambda q & = & \mu \E[\widehat X],\label{rcpsid2}
\end{eqnarray}
in words, the rate of 'natural' migration of susceptible customers matches that of 'external' arrivals
of susceptible customers, and
\begin{eqnarray}
\alpha \E[\widehat X \widehat Y]&  = & \beta \E[\widehat Y],
\label{eq:rmfyth}
\end{eqnarray}
in words, the rate of 'infections' matches that of 'recoveries'.

It follows that
$$
\E[\widehat Y] = \widehat \theta= \frac {\lambda p +\alpha \E[\widehat X \widehat Y]} {\mu + \beta}
= \frac {\lambda p +\beta \widehat \theta} {\mu + \beta},
$$
that is, $\widehat \theta=\frac{\lambda p}{\mu}$, which is 
consistent with (\ref{rcpsid1}).
\\

As for second order relations, when using the fact that $\widehat Y$ is Poisson, one gets
that (\ref{eq:rmf2y}) brings no information, whereas
(\ref{eq:rmf2x}) leads to
\begin{equation}
\label{eq:rmf2xth}
\mu \left(\E[\widehat X^2] -\E[\widehat X] -\E[\widehat X]^2 \right)
= \alpha \left( \E[\widehat X\widehat Y] \E[\widehat X]  
+ \E[\widehat X\widehat Y]  - \E[\widehat X^2\widehat Y] \right) 
\end{equation}
and (\ref{eq:rmf2xy}) to
\begin{equation}
\label{eq:rmf2xyth}
\lambda q \E[\widehat Y]
+ \lambda p \E[\widehat X]
+ \beta \E[\widehat Y] \frac \lambda \mu
= (\beta + 2 \mu) \frac \beta \alpha \E[\widehat Y] 
+ \alpha \E[\widehat X \widehat Y^2].
\end{equation}

\subsubsection{Analytic treatment based on the DOCS ODE solution}
The explicit solution of the open DOCS model is now used to analyze this thermodynamic limit.
We recall that, in the DOCS thermodynamic limit, each station behaves as
an open DOCS reactor with susceptible input rate
\begin{eqnarray*}
\lambda q + \beta \E[\widehat Y] =
\mu \E[\widehat X] + \frac{\beta}{\mu} \lambda p = \lambda q + \frac{\beta}{\mu} \lambda p,
\end{eqnarray*}
with infected input rate
\begin{eqnarray*}
\lambda p + \alpha \E[\widehat X\widehat Y] & = & (\mu + \beta) \E[\widehat Y]
= \frac{\mu+\beta}{\mu} \lambda p,
\end{eqnarray*}
and with susceptible departure rate $\nu=\mu +\beta$.
Hence, it follows from (\ref{eq:odesidpure}) that
\begin{equation}
  \E[\widehat X] =
\frac {\lambda q + \frac {\beta}{\mu} \lambda p} {\mu+\beta +\alpha }
\int_{0}^1
e^{-\frac{\lambda p \alpha^2}{\mu(\mu +\beta +\alpha)^2} (1-t)}
t^{\frac \mu {\mu +\beta +\alpha} +\lambda p\frac{(\mu+\beta)\alpha}{\mu (\mu+\beta+\alpha)^2}-1} dt .
\end{equation}
Since $\mu \E[\widehat X]= \lambda q$, we have
\begin{Prop}
In the DOCS thermodynamic limit, $q$ satisfies the fixed-point relation
\begin{equation}
\label{eq:79}
q = \frac{ q \mu + p \beta }{\mu+\beta +\alpha }
\int_0^1
e^{-\frac{\eta p \alpha^2}{(\mu +\beta +\alpha)^2} (1-t)}
t^{\frac \mu {\mu +\beta +\alpha} +\eta p\frac{(\mu+\beta)\alpha}{(\mu+\beta+\alpha)^2}-1} dt ,
\end{equation}
where the RHS is the analytic expression for $\frac{\E[\widehat X]}\eta$ in this context.
\end{Prop}

\subsubsection{Monotonicity and Convexity Properties}
Here we study the properties of the R.H.S of Equation (\ref{eq:79}) w.r.t. various parameters.

We consider the dependence on $\eta$ first. 
When rewriting the RHS of Equation \eqref{eq:79} as
\begin{align}\label{eq:new1}
\frac{q\mu + p\beta}{\mu+\alpha+\beta}
\int_0^1 e^{ -\eta p \left( \frac{\alpha^2}{(\mu+\alpha+\beta)^2}(1-t)+ \frac{(\mu+\beta)\alpha}{(\mu+\beta+\alpha)^2} |\log t|\right)} 
t^{\frac \mu {\mu +\beta +\alpha} -1} dt,
\end{align}
one gets that this RHS is a strictly decreasing function of $\eta$,
when fixing all other parameters.

The dependence on $\mu$ is more complicated.
The integrand in the RHS of (\ref{eq:79}) is a monotone increasing function of $\mu$
since $\frac{(\mu+\beta)\alpha}{\mu(\mu+\beta+\alpha)^2}$
as a function of $\mu$ is decreasing.
However, the coefficient $\frac{q\mu+p\beta}{\mu+\alpha+\beta}$ is equal
to $q+ \frac{p\beta - q (\alpha+\beta)}{\mu+\alpha+\beta}$ which is either
an increasing or decreasing function of $\mu$, depending on the sign of $p\beta - q (\alpha+\beta)$.

Consider now the dependence in $p$. Clearly, 
\begin{align}\label{eq:new2}
p\to I(p):=\int_0^1 e^{ -\eta \left( \frac{p\alpha^2}{(\mu+\alpha+\beta)^2}(1-t)+ \frac{p(\mu+\beta)\alpha}{(\mu+\beta+\alpha)^2} |\log t|\right)} 
t^{\frac \mu {\mu +\beta +\alpha} -1} 
dt
\end{align}
is a positive and strictly decreasing function of $p$, when fixing all other parameters $\eta,\mu,\alpha,\beta$.
\begin{align*}
I^{''}(p)= \int_0^1 h_1^2(t) \exp (-p h_1(t)) h_2(t) dt 
\end{align*}
is positive, since
\begin{align*}
h_1(t):= \frac{\eta\alpha^2}{(\mu+\alpha+\beta)^2}(1-t) + \frac{\eta(\mu+\beta)\alpha}{(\mu+\beta+\alpha)^2} |\log t| \ \text{and} \ h_2(t):=t^{\frac{\mu}{\mu+\beta+\alpha}-1} 
\end{align*}
 are two strictly positive functions,
and finite since $\lim_{t\to 0} t^r \log t = 0$, for any $r>0$.

Further, the numerator $q\mu + p\beta$ in the prefactor of the RHS of \eqref{eq:new1} may be represented as $\mu + p(\beta-\mu)$ which is strictly positive.
Hence, if $\beta\le \mu$, the function in the RHS
\eqref{eq:79} is strictly decreasing since its derivative w.r.t. $p$ is
\begin{align*}
D(p) := \frac{1}{\mu+\alpha+\beta} \left((\beta-\mu)I(p) + (\mu+
 p(\beta-\mu)) I{'}(p)) \right) <0,
 \end{align*}
 and strictly convex since its second derivative is
 \begin{align*}
  \frac{1}{\mu+\alpha+\beta} \left(2(\beta-\mu)I^{'}(p) + (\mu+
 p(\beta-\mu)) I{''}(p) \right) >0.
 \end{align*}
 In particular,
 \begin{align}\label{eq:RHS0}
 D(0) =
 \frac{\beta}{\mu}-1 -\eta 
\left( \frac{\alpha(\mu+\beta)}{\mu(\mu+\alpha +\beta)} + \frac{\alpha^2}{(\mu+\alpha+\beta)(2\mu+\alpha +\beta)}
\right) \equiv
\frac{\beta}{\mu}-1- \eta \cdot \frac{\alpha}{\mu}\left(1+\frac{\alpha}{2\mu+\beta}\right)^{-1}.
 \end{align}

\subsection{Phase Diagram}

Let 
\begin{equation}
\label{eq:sufcondsid}
\lambda_c^{(d)}:=\frac{\beta\mu}{\alpha}\left(1+\frac{\alpha}{2\mu+\beta}\right),\quad
\eta_c^{(d)} := \frac{l_c^{(d)}}{\mu}= \frac{\beta}{\alpha}\left(1+\frac{\alpha}{2\mu+\beta}\right).
\end{equation}
Here is an analogue of Corollary \ref{Cor10} for DOCS:

\begin{Lem}
\label{lem:tllambdd}
In the DOCS thermodynamic limit, if $\beta\le \mu$,
\begin{itemize}
\item{} 
$p=0$ is the only solution to \eqref{eq:79} $\Leftrightarrow$
$\eta \le \eta_c^{(d)}$;
\item{} 
there are two solutions to Equation \eqref{eq:79} ($p=0$ and $p^*>0$) $\Leftrightarrow$
$\eta > \eta_c^{(d)}$.
\end{itemize}
\end{Lem}

\begin{proof} 
The LHS and the RHS of \eqref{eq:79} coincide at $p=0$, where both are equal to 1.
At $p=1$, the LHS of \eqref{eq:79} is 0, whereas the RHS 
of \eqref{eq:79} is positive.
By \eqref{eq:RHS0}, we have $D(0)<-1$
if and only if $\eta > \eta_c^{(d)}$.
When $\beta\le \mu$,
the RHS of (\ref{eq:79}) is a decreasing and convex function of $p$. Hence,
if the first derivative of the RHS of \eqref{eq:79} at 0 is less than -1, then there is exactly
one positive solution to \eqref{eq:79}, whereas if is more than or equal to -1,
there is no such solution.

In the general case, there is at least one positive solution
if the first derivative of the RHS of \eqref{eq:79} at 0 is less than -1.
\end{proof}

\begin{Thm}
Under the DOCS thermodynamic propagation of chaos ansatz,
if $\eta\le\eta_c^{(d)}$, there is weak extinction
whereas when $\eta>\eta_c^{(d)}$, there is survival.
\end{Thm}

Note that the statement on weak extinction in the last theorem is only proved under the assumption
$\beta\le \mu$. There is numerical evidence that it also holds without this condition.

\subsection{Interpretation of the results}

\paragraph{Monotonicity}
In DOCS, two contradictory phenomena are present
\begin{enumerate}
\item Larger motion rate implies smaller population density which lowers LOCAL epidemic spread;
\item Larger motion rate implies more GLOBAL dissemination of epidemic.
\end{enumerate}
This possibly explains the following facts:
\begin{enumerate}
\item[a] $\eta_c^{(d)}$ is decreasing in $\mu$ (motion rate), that is when fixing population density per
reactor, when one increases motion, the system is less safe.
This is because only 2 acts
whereas 1 is blocked in this setting where $\eta$ is fixed.
\item[b] $\lambda_c^{(d)}$ is increasing in $\mu$, that is when fixing an overall arrival rate in
a reactor, increasing motion leads to a safer system.
In this case 1 and 2 compete
and what our formula shows that 1 (or locality) dominates.
\end{enumerate}

\section{On the Thresholds of the Thermodynamic Limits}
\label{s:ttl}

Here is first a general observation on the connections between the properties 
of the open reactors and the thresholds (w.r.t. $\eta$) of the thermodynamic limits.
We showed that survival in a thermodynamic limit is guaranteed as soon as the derivative at $0$ of the
$p\to p_0(p)$ function of the associated reactor is more than 1.
Classical busy cycle arguments show that the latter derivative can in turn be reinterpreted
as the mean number of susceptible that a single infected customer arriving to the open reactor
infects before leaving the reactor, given that the latter has only susceptible customers upon the
arrival in question. This is precisely
the $R_0$ parameter of epidemiologists. For SIS, the result is stronger: there is
survival if the mean number of susceptible customers that a single infected
customers infects is more than 1 and strong extinction otherwise.
Of course, this mean number depends on the system considered.
For AIR, the evaluation of this mean value is based on an averaging over the geometries of the random environments
that a customer sees. For SIS or DOCS the evaluation takes the random
and dynamic nature of the environment into account. 

Here are now further observations on the thermodynamic limit thresholds in terms
of the last mean value (or $R_0$) interpretation when changing the
scales of the $(\lambda,\mu)$ or that of the $(\alpha,\beta)$ parameters respectively. 

Consider first a family of thermodynamic limit for DOCS indexed by $n$ and assume that $\lambda_n$ and $\mu_n$
both tend to infinity in such a way that the limit of the ratio $\lambda_n/\mu_n$ is a positive
constant $\eta$, whereas $\alpha$ and $\beta$ stay constant. Then $\eta_c^{(d,n)}$
tends to $\eta_c^{(a)}=\beta/\alpha$. Here is a probabilistic interpretation of this fact.
Take a unique infected customer in a large closed network of DOCS reactors,
and assume that the system only contains susceptible except for this infected customer.
As $\mu$ is very large, the latter travels very fast and visits many stations before recovering.
The mean number of customers it infects before recovering is 
$N: = \eta \alpha/\beta$. This formula uses that fact that the recovery time has mean
$1/\beta$, the infection rate is $\alpha$ and some homogenization takes place over
the population of susceptible seen before recovery, due to fast motion.
In this limit ($\lambda$ and $\mu$ high), the condition $\eta < \eta_c^{(d)}$
is then equivalent to the branching criticality condition $N < 1$, which is that of AIR.
The argument should extend to SIS: under the same assumptions,
the threshold for SIS $\eta_c^{(s)}$ should tend to $\eta_c^{(a)}=\beta/\alpha$ as well.

Consider now a family of DOCS thermodynamic limits, also indexed by $n$, and where $\beta_n$ and $\alpha_n$ both
tend to infinity in such a way that the 
limit of the ratio $\beta_n/\alpha_n$ is a positive constant $\kappa$, whereas $\lambda$ and $\mu$ stay constant.
It is easy to see that $\eta_c^{(d)}$ tends to $1+\kappa$. 
The condition $\eta> \eta_c^{(d)}$ then reads $M:= \eta \alpha/(\alpha+\beta) > 1$,
namely again like a branching condition stating that the mean number $M$ of customers infected by
a single infected customers is more than 1. Indeed, a single infected customer arriving to a station
finds in mean $\eta$ susceptible. For a tagged susceptible, start two exponential clocks,
one of parameter $\alpha$ (for its own infection), and one of parameter $\beta$
(for the departure of the infected customer). The chance that the tagged customer gets
infected is hence $\alpha/(\alpha+\beta)$. This justifies the interpretation of $M$.
Now each infected then immediately leaves for another far away queue, where its fate
is the same and independent. This justifies the branching (independence) interpretation. 
Note that $\mu$ has disappeared because it is negligible w.r.t. $\alpha$ and $\beta$. 
This interpretation is specific to DOCS.

\section{Comparison and Numerical Results}
\label{ss:num}

\subsection{Comparison of Plain SIS TL and AIR TL}
The fact that $\eta_c^{(s)}\ge \frac{\beta}{\alpha}$
follows from Lemma \ref{lem:whunc}, which in turn
relies on the conjecture that, in the stationary regime of the SIS thermodynamic limit,
there is a negative correlation between $X$ and $Y$. This negative correlation property is
substantiated by simulation but unproved at this stage.
This can be rephrased in saying that we conjecture that SIS is safer than AIR.
A potential interpretation is as follows: replacing the infection rate $\alpha Y$ by $\alpha {\mathbf E[Y]}$ 
at the same time decreases the infection rate when $Y$ is large and increases it when $Y$ is small.
In the thermodynamic limit, this last fact dominates and makes it more likely for the epidemic to survive
in AIR compared to plain SIS.

\subsection{Comparison of Plain SIS TL and DOCS TL}
In this subsection, we use a mix of simulation and analysis to compare
$\eta_c$ in the thermodynamic limits of SIS and in DOCS.
The main conclusion is that depending on parameters, either SIS or DOCS is safer.
The evaluation of $\lambda_c^{(s)}$
is based Theorems \ref{thm:surv} and \ref{thm:ext}. The simulation of $\eta_c^{(s)}$ is based on
its characterization in Corollary \ref{Cor10} in terms of the derivative at 0 of the
function $p_0(p)$. For DOCS, we use the evaluation of $\lambda_c^{(d)}$ in Lemma \ref{lem:tllambdd}.

Figure \ref{fig:grap-sev1} features a situation where both $\alpha$ and $\beta$ are moderate.
We observe that $\eta_c^{(d)}< \eta_c^{(s)}$ or equivalently SIS is safer than DOCS.

\begin{figure}[!h]
\centering

\includegraphics[width=0.62\linewidth]{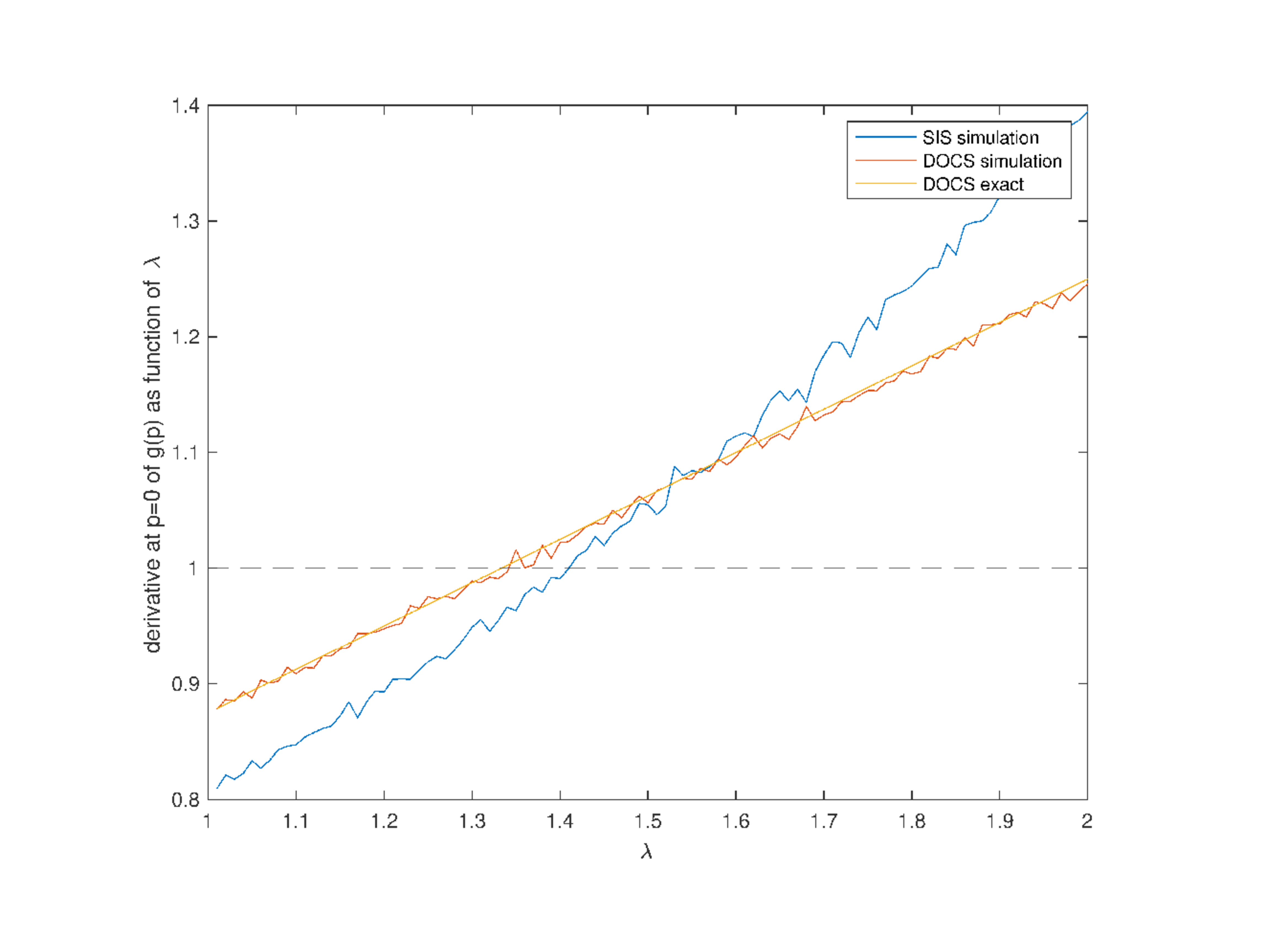}

\caption{Comparison of $\eta_c^{(s)}$ and $\eta_c^{(d)}$ for 
$\mu=1$, $\beta=1$, and $\alpha=1$.
}
\label{fig:grap-sev1}
\end{figure}

Figure \ref{fig:grap-sev2} features a situation where $\alpha$ is large and $\beta$ moderate.
We observe that $\eta_c^{(s)}< \eta_c^{(d)}$ or equivalently DOCS is safer than SIS.
One possible explanation is that although the mean population in each system is the same, in
DOCS, any station with infection present, infected customers leave immediately, which is safer
than the SIS case where they stay and infect all other susceptible customers present in the station.

\begin{figure}[!h]
\centering

\includegraphics[width=0.62\linewidth]{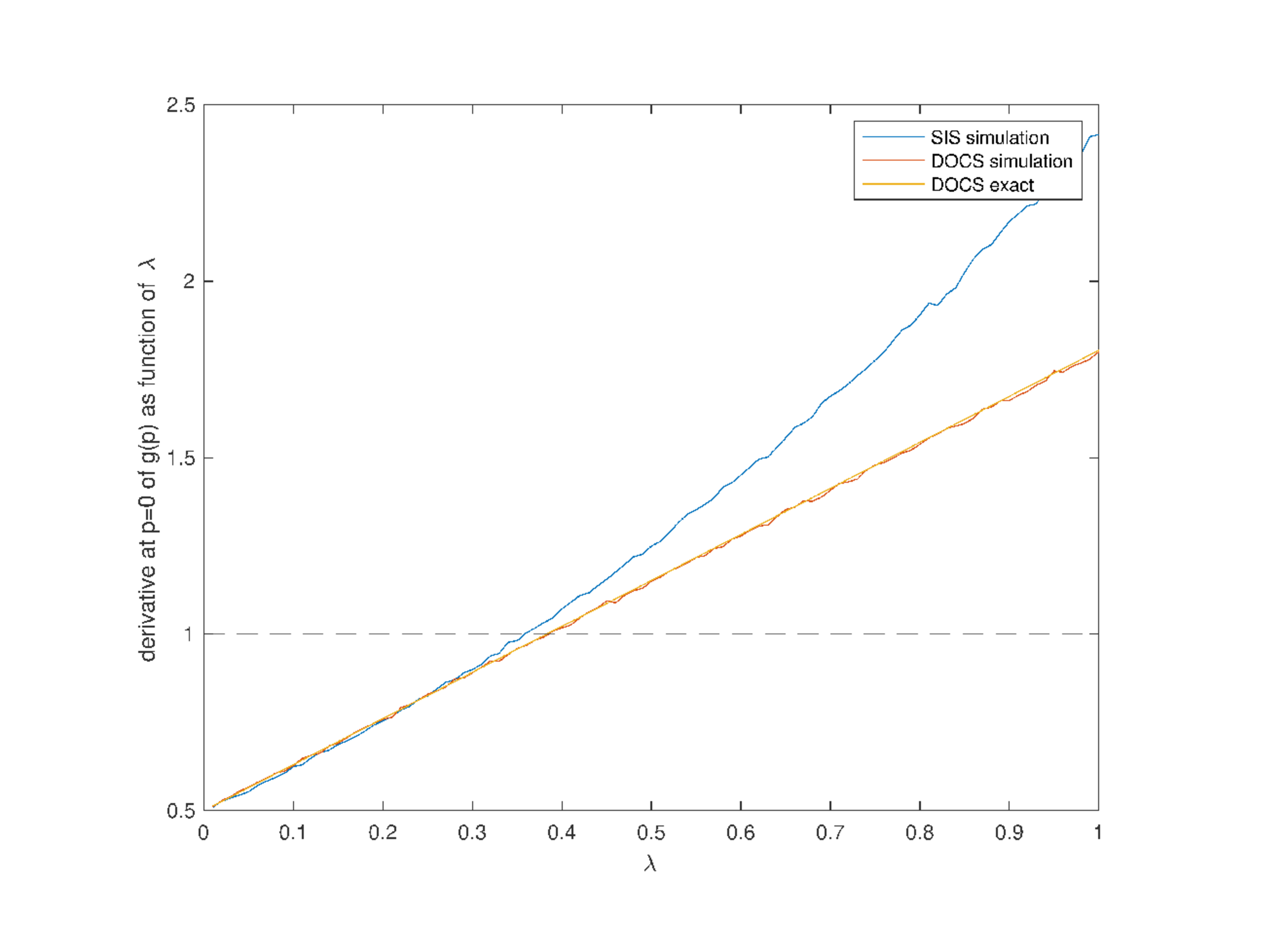}

\caption{Comparison of $\eta_c^{(s)}$ and $\eta_c^{(d)}$ for 
$\mu=1$, $\beta=1$, and $\alpha=20$;
}
\label{fig:grap-sev2}
\end{figure}

Figure \ref{fig:grap-sev3} features a situation where $\beta$ is large and $\alpha$ moderate.
Here SIS is safer than DOCS again. The interpretation is that although the mean number of
individuals per station are the same, DOCS depletes faster any station with many
infected due to their fast recovery.

\begin{figure}[!h]
\centering

\includegraphics[width=0.62\linewidth]{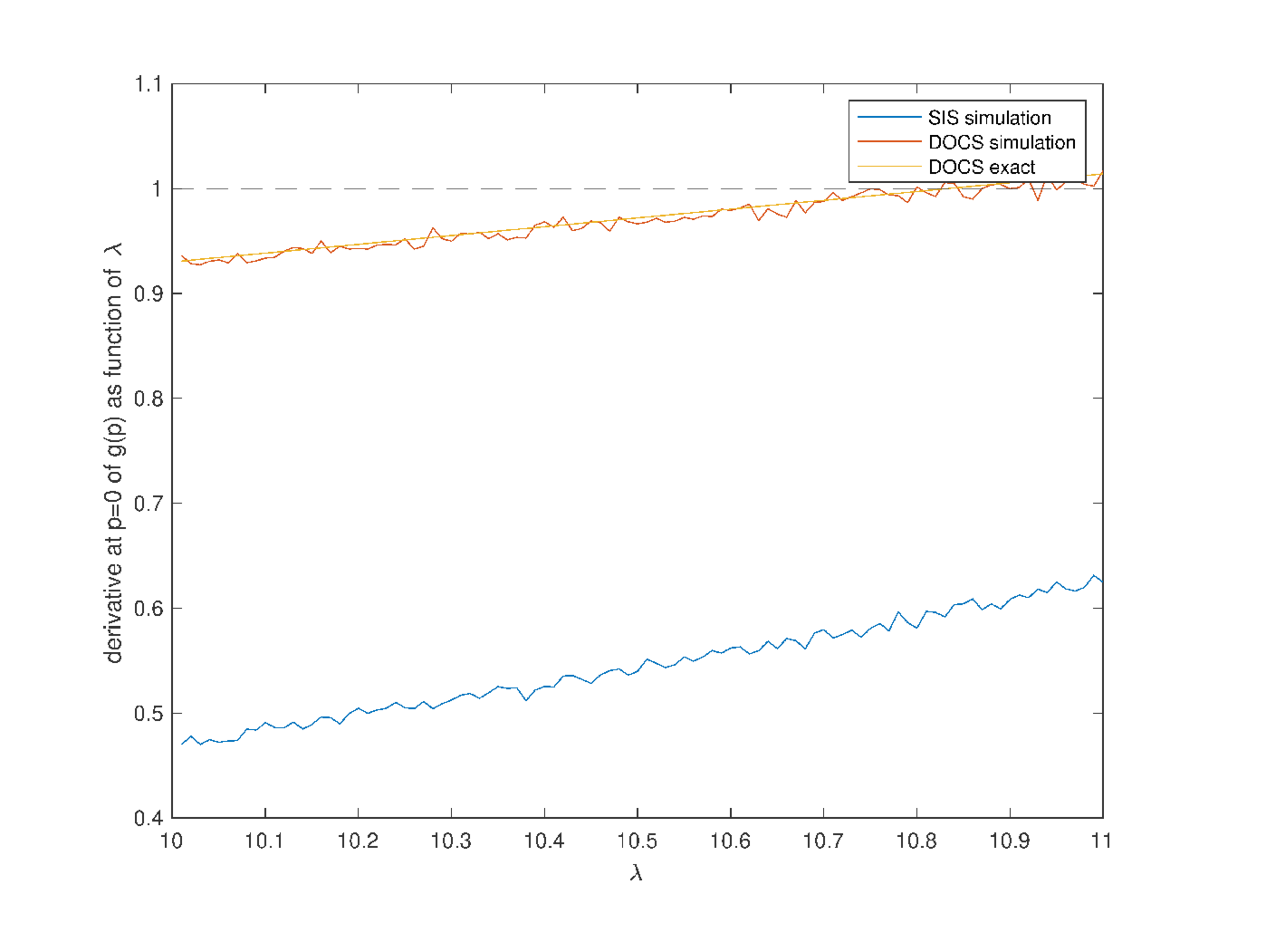}

\caption{Comparison of $\eta_c^{(s)}$ and $\eta_c^{(d)}$ for 
$\mu=1$, $\beta=10$, and $\alpha=1$;
}
\label{fig:grap-sev3}
\end{figure}

\subsection{Comparison of SIS-DOCS TL and AIR TL}
It directly follows from the expressions of $\eta_c$ in both cases
that $\eta_c^{(d)}>\eta_c^{(a)}$. This can be rephrased by saying
that DOCS is safer than AIR.
A possible interpretation is again that extinction happens when $Y$ fluctuates to small values
throughout all reactors of a large system.
By fixing $Y$ to its means as in AIR, in the thermodynamic limit, 
these fluctuations are less likely. This possibly explains why the AIR
system is less safe than the DOCS one. 

\subsection{Phase diagram in $\mu$}

For the SIS reactor and SIS thermodynamic limit, we have characterized the phase diagrams. They turned out to be of the threshold type (survival in one interval, extinction in another). Another interesting question is what such a phase diagram looks like in $\mu$. More specifically, one may be interested in fixing $\eta$, $\alpha$ and $\beta$ and asking for which values of $\mu$ the epidemic survives and for which it does not.

We have shown that, in order to answer such a question, one needs to examine the derivative at $p=0$ of the function $g(p)$. Our findings were based on the monotonicity of this derivative with respect to the relevant parameters.

Numerical evidence shows that such monotonicity with respect to $\mu$ does not always hold.
Indeed, Figure \ref{fig:derivative_vs_mu} contains graphs of the derivative as function of
$\mu$ when other parameters are fixed in two scenarios. On the right, $\eta=3$, $\alpha=5$, $\beta=1$,
and one can observe the absence of monotonicity and convexity. On the left however, when $\eta=\alpha=\beta=1$,
it appears that the derivative is monotone. This evidence suggests that the phase diagram 
in $\mu$ may be more complicated than in the case of other parameters.
This is in line with what was observed in \cite{Baccelli22}.
It is a subject of our future research plans.

\begin{figure}[!h]
\centering
\includegraphics[width=0.45\linewidth]{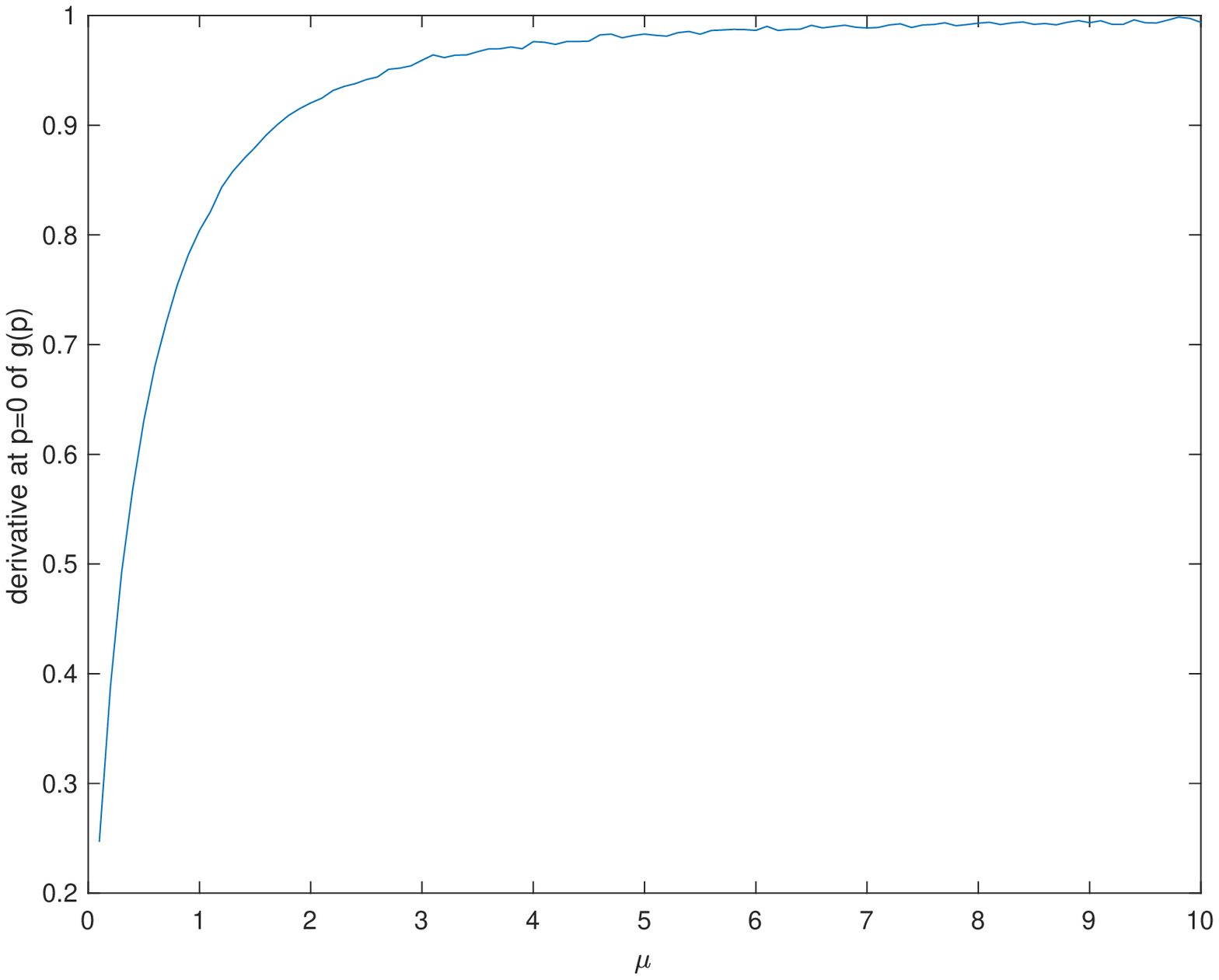}
\includegraphics[width=0.45\linewidth]{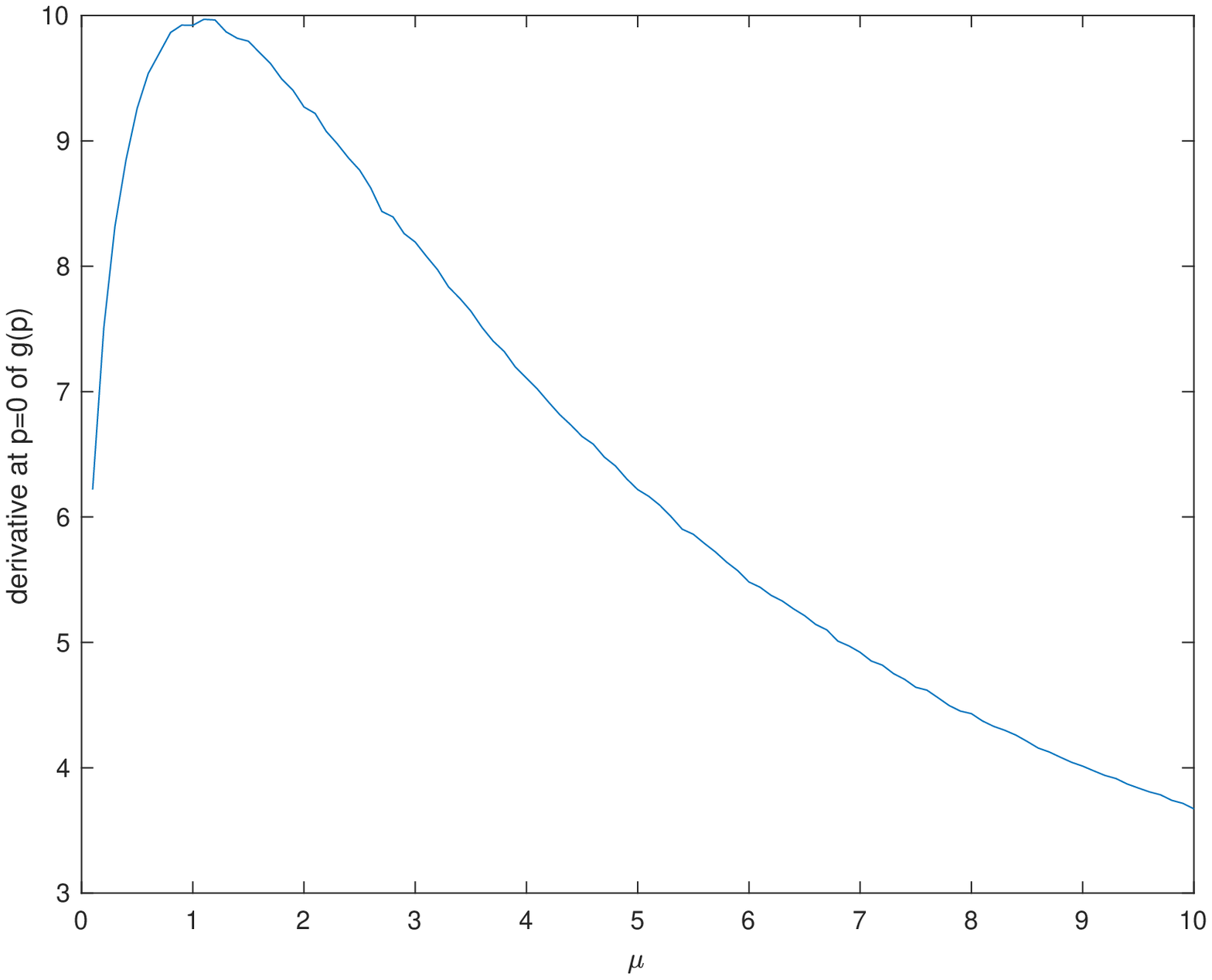}
\caption{Graphs of the derivative at $p=0$ of $g(p)$ in the SIS reactor as function of $\mu$ when $\eta$, $\alpha$ and $\beta$ are fixed. On the left, $\eta=\alpha=\beta=1$, on the right $\eta=3$, $\alpha=5$, $\beta=1$.}
\label{fig:derivative_vs_mu}
\end{figure}

\section{Conclusion and Future Research}
We rigorously proved the structure of the phase diagram of the model of the abstract
under the assumption that the Poisson hypothesis holds.
Future research on the basic model will first consist in
proving the Poisson hypothesis (Ansatz).
Another key objective on the basic SIS model will be to establish the phase diagram
for other parameters than population density.
It would be nice to complement this by computational results like those
established for the AIR and DOCS variants of the model.
Several further variants of the basic model can also be considered,
like SIRS type models where individuals go through
a recovery phase (where they cannot be infected) before being susceptible again.
We will also study further spatial SIS dynamics.
Further basic queueing models, for instance finite capacity queues, can also be considered.
In fact, such epidemics can be devised on virtually all queueing network models
of the literature.

\section{Appendix}

\subsection{Rate Conservation}
\label{app-sec-crc}
One can rewrite (\ref{eqsecy}) as
\begin{equation*}
\lambda p \E [Y]
+ \alpha \E[ XY^2] =
\lambda p \E [Y]
+ \alpha \E[XY] \E[ \frac{\alpha XY}{\alpha \E[XY]} Y ]
=(\mu +\beta) \E[Y(Y-1)].
\end{equation*}
Let $I$ denote the stationary infection epoch point process.
Its intensity is $a_i = \alpha \E[XY]$ and by Papangelou's theorem \cite{BB03}
\begin{equation}
\label{eqnotess}
\E[ \frac{\alpha XY}{\alpha \E[XY]} Y ] = \E^0_{I}[ Y],
\end{equation}
with $\E^0_{I}$ denoting expectation w.r.t. the Palm probability w.r.t. $I$ \cite{BB03}. 
Hence the LHS of the last equality is
\begin{equation*}
\lambda p \E [Y]
+ a_i \E^0_{I}[ Y],
\end{equation*}
in which we recognize the half of the increase rate of $\E[Y(Y-1)]$.
Let $M_i$ denote the point process of recovery or departures of infected customers.
The RHS of (\ref{eqnotess}) can be rewritten as
$$(\mu +\beta) \E[Y] \E[ \frac{(\mu +\beta)Y}{(\mu +\beta) \E[Y]} (Y-1) = d_i \E^0_{N} [Y-1],$$
with $d_i$ the intensity of $M_i$ and the same Palm probability notation as above.
This is twice the decrease rate of $Y(Y-1)$.
Hence (\ref{eqsecy}) is nothing else than the rate conservation principle (RCP) for 
$\frac 1 2 Y(Y-1)$. 

The RCP for $Y^2$ in turn reads
\begin{equation}
\label{eqsecy-rcp}
\lambda p \E [2Y+1]
+ \alpha \E[ XY(2Y+1)]
=(\mu +\beta) \E[Y(2Y-1)].
\end{equation}

Similarly, we can rewrite (\ref{eqsecx}) as
\begin{equation*}
\lambda q \E [X] +\beta \E[XY] =\alpha \E[ XY(X-1)] +\mu \E[X(X-1)]
\end{equation*}
or equivalently
\begin{equation*}
\lambda q \E [X] +\beta \E[XY] =\alpha \E[XY] \E[\frac{\alpha XY}{\alpha \E[XY]}(X-1)]
+\mu \E[X] \E[\frac{\mu X}{\mu \E[X]}(X-1)].
\end{equation*}
By the same arguments,
the RHS is half the decrease rate of the second factorial moment of $X$ due
to infection and departures. The LHS is half the increase rate of the same quantity due 
to arrivals of susceptible and recoveries of infected.
This equation is hence the RCP for $\frac 1  2 X(X-1)$.
The RCP for $X^2$ reads
\begin{equation}
\label{eqsecx-rcp}
\lambda q  \E [2X+1] +\beta \E[Y (2X+1)] =\alpha \E[ XY(2X-1)] +\mu \E[X(2X-1)].
\end{equation}

When differentiating the PDE w.r.t. $x$ and $y$ (or $y$ and $x$) and taking $x=y=1$, we get
\begin{equation}
\label{eqsecxy-rcp}
\lambda p \E [X] + \lambda q\E [Y] +
\beta \E [Y(Y-X-1)]+ \alpha(\E[ XY(X-Y-1])
= 2\mu\E [XY].
\end{equation}
It is easy to check that injecting the two relation of the last lemma in
the last equation does not give anything new. It just confirms that
$X+Y$ is Poisson $\lambda/\mu$.
It is easy to check that this is in fact the RCP for $XY$. 
\begin{Rem}
\label{rem4}
Using the Palm interpretation given above, and PASTA,
it is easy to check that  (\ref{eqsecy}) can be rewritten as
\begin{equation}
\label{eqsecy-ba}
\lambda p \E_{A_i} [Y^-]
+ a_i \E_{I}[ Y^-] =
\mu_i \E_{D_i}[Y^+] 
+a_s \E_{R}[Y^+], 
\end{equation}
with $\mu_i=\mu\E[Y]$ the exogenous departure rate in this queue,
$a_s=\beta\E[Y]$ the recovery rate, and $a_i=\alpha\E[XY]$ the infection rate.
This is nothing else as the classical property that, in the "infected queue", the Palm expectation 
of the number of customers just before arrivals coincides with the Palm expectation
of the number of customers just after departures.
Similarly, (\ref{eqsecx}) reads
\begin{equation}
\label{eqsecx-ba}
\lambda q \E_{A_s} [X^-] +a_s \E_{R}[X^-] = a_i \E_{I}[ X^+] +\mu_s \E_{D_s}[X^+],
\end{equation}
with similar notation.
\end{Rem}

\begin{Cor}\label{cor1} The following relations hold:
\begin{equation}
\lambda \left(q +\frac \beta \mu\right)
\left(1 +\frac \beta \alpha\right)
+\left(\lambda q -\frac \beta \alpha(\mu+\alpha+\beta)\right) \E [X]
-\mu \E[X^2] 
=\alpha \E[ X^2Y]
\end{equation}
and
\begin{eqnarray}
& & \hspace{-1.6cm} \frac \lambda \mu 
\left(\lambda p +\frac {\mu +\beta} {\mu\alpha}  (2\mu(\mu q+\beta) -\lambda \alpha)
\right)
-\left(\lambda p +\mu+\beta +\frac 2 \alpha (\mu+\beta)^2\right) \E [X] \nonumber \\
& & +(\mu+\beta) \E[X^2] =-\alpha \E[ XY^2].
\end{eqnarray}
\end{Cor}
\begin{proof}
Use (\ref{eq2}) to eliminate $\E[XY]$ in each of the relations of Lemma \ref{lemsec}.
Use also the relations $\E[X]+\E[Y]= \frac \lambda \mu $ and
$$ \E[X^2] + 2\E[XY] + \E[Y^2] = \frac {\lambda^2} {\mu^2} + \frac \lambda \mu $$
to get the second relation.
\end{proof}

\subsection{Proofs of Lemmas \ref{lem:pop}, \ref{lem:unic}  and \ref{lem8}}
\label{secap1}

\paragraph{Proof of Lemma \ref{lem:pop}}
The proof includes 4 steps: monotonicity, concavity, strict concavity, and differentiability.

Below, to make the proof more transparent, we use the notation $N_I(t)$
instead of $X(t)$ and $N_S(t)$ instead of $Y(t)$. In addition, the properties
are expressed w.r.t. $\lambda=\eta\mu$ rather than $\eta$.
Since $\mu$ is a constant, this is equivalent.

\paragraph{(1) Monotonicity.} 

Consider the typical busy cycle. Assume that the first customer arrives at an empty system at
time $0$ and let $T>0$ be the first time when the system becomes empty again.
We have to show that the mean number of infected customers that depart
from the system within time interval $[0,T]$ is an increasing function of $p$.

For that, we consider two models, with input probabilities of being infected $p$ and $\widehat{p}>p$.
For the input probability $p$ and for $0\le t \le T$, we denote by $N(t)$ the total 
number of customers at time $t>0$,
by $N_I(t)$ the number of infected customers and by $N_S(t)$ the number of susceptible customers.
Clearly, $N(t)=N_I(t)+N_S(t)$. Let
$\widehat{N}(t)$, $\widehat{N}_I(t)$ and $\widehat{N}_S(t)$ by the corresponding
processes related to probability $p'$, with 
$\widehat{N}(t)=\widehat{N}_I(t)+\widehat{N}_S(t)$ for $t\ge 0$.

We produce a coupling of two the models such that, for all $0\le t\le T$,
$N(t)=\widehat{N}(t)$ a.s., while $\widehat{N}_I(t)\ge N_I(t)$ and, therefore, $\widehat{N}_S(t)\le N_S(t)$.
Thus, for the total numbers of departures of infected customers within the first busy cycle,
we get $D_I(T) \le \widehat{D}_I(T)$ a.s. Moreover, we will show that ${\mathbf P}(D_I(T) < \widehat{D}_I(T))>0$.

Let $t_n, n\ge 1$ be exponential-$\lambda$ random variables, $\sigma_{n,j}, n,j\ge 1$  exponential-$\mu$
random variables, $a_{n,i,j}$ exponential-$\alpha$ random variables, $b_{n,j}$ exponential-$\beta$
random variables, and $u_{n}$ have uniform distribution in the interval $(0,1)$.
We assume all these r.v.'s to be mutually independent.

We introduce embedded epochs $T_0=0<T_1< T_2 < \ldots < T_{\psi}=T$, where $\psi$ is a random
natural number representing the total number of events (jumps) that occur in both systems.
This is not restrictive as some of these events will be fictitious in either system.
Both processes $(N_S(t),N_I(t))$ and $(\widehat{N}_S(t),\widehat{N}_I(t))$ are piecewise constant and may
make jumps at time instants $T_n$ only. We assume the processes to be right-continuous, 
and customers to be numbered in each group at any time, in a way to be made precise in due time.

At time $T_0=0$, we let 
\begin{align*}
N_S(0)={\mathbf I} (u_1>p) = 1 -N_I(0) \ \ \mbox{and}
\ \ 
\widehat{N}_S(0)={\mathbf I} (u_1>\widehat{p}) = 1 -\widehat{N}_I(0)
\end{align*}
and
\begin{align*}
D_S(0)=D_I(0)=\widehat{D}_S(0)=\widehat{D}_I(0)=0.
\end{align*}
Then we evaluate the values of the processes recursively: that is,
we produce both time $T_{n+1}$ and their values at time $T_{n+1}$ 
given time $T_n$, their values at time $T_n$, and the values of $t_{n+1}$, $u_{n+1}$,  
$\{\sigma_{n+1,j}\}_{j\ge 1}$, $\{a_{n+1,i,j}\}_{i,j\ge 1}$,
$\{b_{n+1,j}\}_{j\ge 1}$.

Assume we have introduced both processes up to time $T_n$ and showed that that $\widehat{N}_I(T_n)\ge N_I(T_n)$ and $\widehat{N}(T_n)=N(T_n)$. 
Assume that $\sigma_{n+1,j}, 1\le j \le N(T_n)$ are the remaining
service times of all present customers, where customers numbered $1,\ldots,N_I(T_n)$
are infected in both systems, customers numbered
$N_I(T_n)+1,\ldots,\widehat{N}_I(T_n)$ are susceptible in the first systems and
infected in the second system, and customers numbered $\widehat{N}_I(T_n)+1, \ldots,N(T_n)$
are susceptible in both systems.
Assume that, for $j=N_I(T_n)+1,\ldots,\widehat{N}_I(T_n)$, the random variables $b_{n+1,j}$ 
represent the recovery times for corresponding infected customers in the second system
and for $j=1,\ldots, N_I(T_n)$, the recovery times for customers that are infected in both systems.
Assume that if customer $j$ is infected and customer $i$ is susceptible at time $T_n$,
then $a_{n+1,j,i}$ is the instant when $j$ may infect $i$.

Consider the random sets of integers 
\begin{align*}
C_1(n) = \{j:   1\le j \le N_I(T_n)\},  
C_2(n) = \{j:  N_I(T_n)< j \le \widehat{N}_I(T_n)\},  
C_3(n) = \{j:  \widehat{N}_I(T_n)< j \le N(T_n)\}.
\end{align*}
Let, for $k=1,2,3$, 
\begin{align*}
\Sigma_{n+1,k} = \min_{j\in C_k(n)} \sigma_{n+1,j} \ \ 
\mbox{and}
\ \  B_{n+1,k} = \min_{j\in C_k(n)} b_{n+1,j},
\end{align*}
where $\min_{\emptyset} = \infty$, by convention.
Next, for $k,l=1,2,3$,  let
\begin{align*}
A_{n+1,k,l} = \min_{j\in C_k(n), i\in C_l(n)} a_{n+1, j,i}.
\end{align*}
Finally, let
\begin{align}
\theta_{n+1} = \min (t_{n+1}, \min_{1\le k \le 3} \Sigma_{n+1,k},
\min_{1\le k \le 2} B_{n+1,k},
\min (A_{n+1,2,3}, A_{n+1,1,3},A_{n+1,1,2})).
\end{align}
Therefore, $\theta_{n+1}$ is the next time instant when something happens in any of the systems,
either the arrival of a new customer or departure of one of present customers or recovery
of one of present customers or infection of one of a present customers by another one.

Then
\begin{align*}
T_{n+1} = T_n+\theta_{n+1}.
\end{align*}

If $\theta_{n+1}=t_{n+1}$, i.e. a new customer arrives. Then 
\begin{align*}
D_S(T_{n+1})=D_S(T_n), D_I(T_{n+1})=D_I(T_n),
\widehat{D}_S(T_{n+1})=\widehat{D}_S(T_n), \widehat{D}_I(T_{n+1})=\widehat{D}_I(T_n)
\end{align*}
and $N(T_{n+1})=\widehat{N}(T_{n+1})=N(T_n)+1$, 
\begin{align*}
N_S(T_{n+1}) &=N_S(T_n)+{\mathbf I}(u_{n+1}>p), \ \
\widehat{N}_S(T_{n+1}) =\widehat{N}_S(T_n)+{\mathbf I}(u_{n+1}>\widehat{p}),\\
N_I(T_{n+1}) &=N_I(T_n)+{\mathbf I}(u_{n+1}\le p), \ \
\widehat{N}_I(T_{n+1}) =\widehat{N}_I(T_n)+{\mathbf I}(u_{n+1}\le \widehat{p}).
\end{align*}

If $\theta_{n+1}= \min_{1\le k\le 3} \Sigma_{n+1,k}$,
then one of present customers departs.
If $\Sigma_{n+1,1}$ is the smallest among the three,
then there is departure of a customer that is infected in both systems. Therefore, 
\begin{align*}
D_I(T_{n+1})=D_I(T_n)+1, \ \ 
\widehat{D}_I(T_{n+1})=\widehat{D}_I(T_n)+1, \ \ 
D_S(T_{n+1})=D_S(T_n), \ \ 
\widehat{D}_S(T_{n+1})=\widehat{D}_S(T_n)
\end{align*}
and
\begin{align*}
N_I(T_{n+1})=N_I(T_n)-1, \ \ 
\widehat{N}_I(T_{n+1})=\widehat{N}_I(T_n)-1, \ \ 
N_S(T_{n+1})=N_S(T_n), \ \ 
\widehat{N}_S(T_{n+1})=\widehat{N}_S(T_n).
\end{align*}
By the symmetry, if $\sigma_{n+1,3}$ is the smallest, then there is departure of a customer that is susceptible in both systems.  Therefore,
\begin{align*}
D_S(T_{n+1})=D_S(T_n)+1, 
\widehat{D}_S(T_{n+1})=\widehat{D}_S(T_n)+1,
D_I(T_{n+1})=D_I(T_n), 
\widehat{D}_I(T_{n+1})=\widehat{D}_I(T_n)
\end{align*}
and
\begin{align*}
N_S(T_{n+1})=N_S(T_n)-1, 
\widehat{N}_S(T_{n+1})=\widehat{N}_S(T_n)-1,
N_I(T_{n+1})=N_I(T_n), 
\widehat{N}_I(T_{n+1})=\widehat{N}_I(T_n).
\end{align*}
Next, if $\Sigma_{n+1,2}$ is the smallest, then this means that the strict inequality
$N_I(T_n)<\widehat{N}_I(T_n)$ holds and there is a departure of a customer that is susceptible
in the first system and infected in the second. Then 
\begin{align*}
D_I(T_{n+1})=D_I(T_n), 
\widehat{D}_I(T_{n+1})=\widehat{D}_I(T_n)+1,
D_S(T_{n+1})=D_S(T_n)+1, 
\widehat{D}_S(T_{n+1})=\widehat{D}_S(T_n)
\end{align*}
and
\begin{align*}
N_I(T_{n+1})=N_I(T_n), 
\widehat{N}_I(T_{n+1})=\widehat{N}_I(T_n)-1,
N_S(T_{n+1})=N_S(T_n)-1, 
\widehat{N}_S(T_{n+1})=\widehat{N}_S(T_n).
\end{align*}
One can see that in all three cases
\begin{align}\label{induction1}
N_I(T_{n+1})\le \widehat{N}_I(T_{n+1}) 
\ \ \mbox{and}
\ \ 
D_I(T_{n+1})\le \widehat{D}_I(T_{n+1}).
\end{align}

Similarly, if $\theta_{n+1} = \min_{k=1,2} B_{n+1,k}$,
then the numbers of departures do not change
and either there is a simultaneous recovery of a customer that was infected in both systems or
there was a customer that was susceptible in the first system
and infected in the second, it has recovered in the second system
and nothing has changed in the first system.
Again, one can see that the needed inequalities \eqref{induction1} continue to hold.

Finally, if $\theta_{n+1} = \min (A_{n+1,2,3}, A_{n+1,1,3},A_{n+1,1,2})$,
then the number of departures stays the same and there are again three scenarios.
If $A_{n+1,2,3}$ is the smallest among the three, then the set $C_2(n)$ of customers that are susceptible 
in the first system and infected in the second system at time $T_n$ 
is non-empty, one of them has infected one of susceptible customers in the second system,
and nothing has changed in the first system. Therefore, the number of susceptible customers
in the second system decreases and the number of susceptible customers in the first system stays
the same. Thus, the needed inequalities \eqref{induction1} continue to hold.
If $A_{n+1,1,2}$ is the smallest, then the set $C_2(n)$ is non-empty again, so we
have strict inequality $N_S(T_n)>\widehat{N}_S(T_n)$ and,
at time $T_{n+1}$, one of customers from this set
becomes infected in the first system and nothing changes in the second system. Therefore,
\begin{align*}
N_S(T_{n+1})=N_S(T_n) \le \widehat{N}_S(T_n)-1 =
\widehat{N}_S(T_{n+1}),
\end{align*}
as required. 
If $A_{n+1,1,3}$ is the smallest, then one of the customers that is healthy at
time $T_n$ in both systems becomes infected in both system. So the required inequalities continue to hold.

This completes the proof of the fact the $D_I(t)\le \widehat{D}_I(t)$ at any time $0<t<T$ and, in particular,
$D_I\equiv D_I(T) \le \widehat{D}_I\equiv \widehat{D}_I(T)$ a.s. It is left to show that the inequality may
be strict with positive probability. However, this is almost obvious:
\begin{align*}
\{D_I < \widehat{D}_I\} \supseteq \{D_I=0, \widehat{D}_I=1\} = 
\{\sigma_{1,1}<t_1\} \cap
\{ p<u_1<\widehat{p}\}, 
 \end{align*}
where the events on the right are independent and of positive probabilities.

\begin{Rem}
\label{rem24}
The monotonicity property discussed above can be extended in two ways: 
\begin{enumerate} 
\item Rather than comparing the two systems in a busy cycle, one compares them
over the whole time half axis. For this, one has to start the two systems 
at time $T_0=0$ with the same total population $N(T_0)=\widehat{N}(T_0)$ 
and with more infected customers in the dominating system than in the 
dominated one, i.e., with $N_I(T_0)\le \widehat{N}_I(T_0)$ a.s. 
Based on the same arguments, one
proves by induction over the overall jump times $\{T_n\}$ 
that in the coupling described above, for all $t$,
$N(t)= \widehat{N}(t)$ a.s. and
$N_I(t)\le \widehat{N}_I(t)$ a.s. 
\item The setting of 1. above can be extended to the situation where
rather than having constant fractions of the Poisson arrivals, $p$ and $\widehat p$, that are infected
in the two systems, with $p\le \widehat p$, one has deterministic 
time-varying fractions $p(t)$ and $\widehat p(t)$ of the Poisson arrivals 
which are infected in the two systems, with $p(.)\le\widehat p(.)$. Then again, 
if $N_I(T_0)\le \widehat{N}_I(T_0)$ a.s., then, in the same coupling,
$N(t)= \widehat{N}(t)$ a.s. and
$N_I(t)\le \widehat{N}_I(t)$ a.s. for all $t$.
\end{enumerate} 
\end{Rem}

\paragraph{(2) Concavity.}
We now introduce a slightly different coupling that allows us to consider simultaneously
models with different parameters.

From now, we do not enumerate customers. Instead, we consider a countable number ${\cal Z}$
of ``locations" for them that coincide with ``servers". More precisely,
we assume that an arriving customer chooses an empty server at random and stays
there until its departure from the system. 

We continue to consider a single busy cycle on the time interval $[0,T]$.
We again denote by $T_0<T_1<T_2<\ldots$ the time instants when the state of the systems {\it may} change.
Hence, the state of the system $Z(t)$ at time $t$ is a collection of pairs
$\{ (z,c_z(t)), z\in{\cal Z}\}$, where $z$ is a server and $c_z$ is its
``color''. Here $c_z(t) ={\bf n}$ (where ``${\bf n}$'' means ``no color'') if server $z$ is
empty at time $t$ and $c_z(t)$ has one of several other colors, otherwise. 
The meaning of these colors will come in due time. Let 
\begin{align*}
{\cal N}(t) = \{z : c_z(t) \ne {\bf n} \}
\end{align*}
be the set of occupied servers, and  
\begin{align*}
N(t) = \sum_{z\in {\cal z}} {\mathbf I} (c_z(t)\ne{\bf  n})
\end{align*} 
its cardinality.

We introduce the following sequences of random variables:
\begin{itemize}
\item $\{u_n\}_{n\ge 0}$ are uniformly distributed in $(0,1)$; 
\item $\{t_n\}_{n\ge 1}$ are exponential-$\lambda$;
\item $\{\sigma_{n,z}\}_{n\ge 0, z\in{\cal Z}}$ are exponential-$\mu$;
\item $\{b_{n,z}\}_{n\ge 0, z\in {\cal Z}}$ are exponential-$\beta$;
\item $\{a_{n,z,v}\}_{n\ge 0, z,v\in {\cal Z}, z\ne v}$ are exponential-$\alpha$.
\end{itemize}
We assume that all these random variables are mutually independent.

Let $T_0=0$. Introduce recursively $\{\theta_k\}$ and $T_k=\sum_1^k \theta_j$. Denote ${\cal N}_k =  {\cal N} (T_k+0)$.
For $n\ge 1$, assuming that $T_{n-1}$ is defined, we let  
$\Sigma_n = \min_{z\in {\cal N}_{n-1}} \sigma_{n,z}$,
$B_n =   \min_{z\in {\cal N}_{n-1}} b_{n,z}$, 
$A_n =  \min_{z,v\in {\cal N}_{n-1}, z\ne v} a_{n,z,v}$.
Then let 
\begin{align*}
\theta_n = \min \{ t_n, \Sigma_n,A_n,B_n\} \ \ \mbox{and} \ \ 
T_{n} = T_{n-1}+\theta_n.
\end{align*}

Now we are ready to define a ``3-color'' process.
Let $p\ge 0, q\ge 0, r>0$ be numbers that sum up to 1, $p+r+q=1$. 
We assume that each arriving customer gets the red color with probability $p$, 
the magenta color with probability $r$ and the green color with probability $q$. 
It occupies a server (that gets the same color as the customer).

We assume the process to be piecewise constant between time instants $T_n$. 
We introduce coloring by induction.

At time $T_0=0$, we color the arriving customer magenta if $u_0\le r$, red if $r<u_0\le r+p$ and green
if $u_0>r+p$, and assume that the customer keeps its color within time interval $(T_0,T_1]$.

Assume that coloring is done up to time $T_n$.
Proceed with coloring for the time interval $(T_n,T_{n+1}]$.

If $\theta_n=t_n$, then we place an arrived customer at one of the empty servers and color
the arrived customer magenta if $u_n\le r$, red if $r<u_n\le r+p$ and green if $u_n>r+p$.

If $\theta_n=\sigma_{n,z}$ for some $z\in {\cal N}_n$,
then customer $z$ leaves the system and the corresponding server becomes idle (`no color').

If $\theta_n = \beta_{n,z}$ for some $z\in {\cal N}_n$,
then customer $z$ becomes green whatever color it had before.

If $\theta_n = A_n$, then we have to consider three cases.
At time $T_n$, let ${\cal N}_{n,m}, \  {\cal N}_{n,r}, \ {\cal N}_{n,g}$
be the sets of magenta, red and green customers, where clearly 
\begin{align*}
{\cal N}_{n} = {\cal N}_{n,m} \cup {\cal N}_{n,r} \cup {\cal N}_{n,g}.
\end{align*} 
Let $N_{n,m}, \ N_{n,r}$ and $N_{n,g}$ be the corresponding cardinalities. 
If $\theta_n = a_{n,z,v}$ for some $z\in {\cal N}_{n,g}$,
then (since green means susceptible) there is no new infection and
all colors stay the same (one can say that this is a ``false coloring'').

If $\theta_n = a_{n,z,v}$ for some $z\in {\cal N}_{n,r}$,
then customer $v$ (and the corresponding server) gets red, whatever color it had before.

Finally, if $\theta_n = a_{n,z,v}$ for some $z\in {\cal N}_{n,m}$,
then customer $v$ (and the corresponding server) gets magenta if it had either
magenta or green before -- or keeps the red color if it was red before.

Thus, we have defined the 3-color dynamics within the busy cycle.

As a result of the proposed coupling construction,
one can observe the following. 
\begin{enumerate}
\item If we do not distinguish (``merge'') the green and magenta colors (and recolor them as ``new green''),
then we get the ``two-color'' susceptible-infected  model considered earlier, with the probability
of green arrivals $r+q$
and the probability of red arrivals $p$; call this model the $(p, r+q)$ model.
\item If we do not distinguish (``merge'') the magenta and red colors (and recolor them as ``new red''),
then we get the ``two-color''  susceptible-infected  model considered earlier,
with the probability of green arrivals $q$
and the probability of red arrivals $r+p$; call this model the $(r+p,q)$ model.
\end{enumerate}
This means that, for $0\le t\le T$, the random variable  $N_m(t)$ represents the ``excess''
of the number of `red' customers in the $(r+p,q)$ model in comparison with the $(p,r+q)$ model. 

Assume now that $q>0$ and consider the 3-color
$(r,\widehat{p},\widehat{q})$-model with $\widehat{p}>p$ and $r+\widehat{p}+\widehat{q}=1$.
Let $\widehat{\cal N}_m(t)$ be the set of magenta customers at time $t$ in this model and $\widehat{N}_m(t)$ its cardinality.
Then straightforward induction arguments show that
$\widehat{\cal N}_m(t) \subseteq {\cal N}_m(t)$ and, therefore,
$\widehat{N}_m(t)\le N_m(t)$ a.s.,  for any $0\le t \le T$.
In turn, this implies that the total number of departures within the first cycle of
magenta customers in the corresponding systems,
call them $D_m(T)$ and $\widehat{D}_m(T)$, satisfy 
\begin{align*}
D_m(T) \ge \widehat{D}_m(T) \ \ \mbox{a.s.}
\end{align*}
Recall the notation $p_o=g(p)$. Then
\begin{align*}
g(p) = \frac{{\mathbf E} D_g(T)}{{{\mathbf E} D(T)}}, \ \ 
g(p+r) = \frac{{\mathbf E}(D_g(T)+D_m(T))}{{{\mathbf E} D(T)}}, \ \ 
g(\widehat{p}) = \frac{{\mathbf E} \widehat{D}_g(T)}{{{\mathbf E} D(T)}}, \ \ 
g(\widehat{p}+r) = \frac{{\mathbf E}(\widehat{D}_g(T)+\widehat{D}_m(T))}{{{\mathbf E} D(T)}}, 
\end{align*}
where. as before, 
${{\mathbf E} D(T)}$ is the mean number of customers served in the first busy cycle (which is the same as the total number of departures within the cycle).
Since
\begin{align*}
g(p+r) - g(p) 
=
\frac{{\mathbf E} D_m(T)}{{{\mathbf E} D(T)}}
\ge 
\frac{{\mathbf E} \widehat{D}_m(T)}{{{\mathbf E} D(T)}}
=
g(\widehat{p}+r) - g(\widehat{p}),
\end{align*}
we get the required concavity. 

\paragraph{(3) Strict concavity.}
It is enough to consider the latter coupling construction and to justify  that, for any $r>0$ and
any $0\le p <\widehat{p} \le 1-r$,
${\mathbf P} (D_m(T)>\widehat{D}_m(T))>0$.
One can see that this strict inequality holds on the following event of a strictly positive probability:
\begin{align}\label{strictconv}
\{ u_0\le r, \ \ \theta_1=t_1, \ \ u_1\in (r+p, r+\widehat{p}),
\theta_2 = a_{2,2,1}, \ \theta_3=\sigma_{3,0}, \ \ 
\theta_4 = \sigma_{4,1}\}.
\end{align}
The latter means that the following sequence of events occurs:
\begin{itemize}
\item the first customer gets the magenta
color in both 3-color systems;
\item then the second customer appears and gets the green color in the first system and red in the second;
\item then the second customer attempts to recolor the first one, this does not work in the first system and we continue to have
1 green and 1 magenta, while the attempt is successful in the second system, and we get two reds;
\item then the first customer leaves the system and then the second leaves the system, this ends the busy cycle.
\end{itemize}

One can see that $D_m(T)=1>\widehat{D}_m(T)=0$ on the event
\eqref{strictconv}.
 
\paragraph{(4) Differentiability.} 
Let $h>0$ be small.  For $p>0$, introduce a 4-color model with 
green, violet, magenta and red colors. Consider the second coupling construction
introduced in the {\bf `convexity'} subsection. If we have an arrival at time 
 $T_n$, then the new customer gets green color
 if $u_n<p-h$, violet color if $p-h\le u_n < p$,
 magenta color if $p\le u_n < p+h$, and red color if
 $u_n \ge p+h$. Let $A_v$ be the event that there is only one violet arrival and no magenta arrivals in the (first) busy cycle, and $A_m$ that there is only one magenta arrival and no violet arrivals. These events have equal probabilities that are of order $ch+o(h)$, where $c={\mathbf E} \nu$ and $\nu$ is the mean number of arrivals within a busy cycle; there appear only (at most) three colors (green, violet, and red) on the event $A_v$ and, similarly, at most three colors (green, magenta, and red) on the event $A_m$. Moreover, the dynamics of the process on the events $A_v$ and $A_m$ are identical, with the obvious swap of violet and magenta customers.
 Therefore, the mean number ${\mathbf E} (D_v(T)\ | A_v))$ of violet departures given the event $A_v$ coincides with the mean number ${\mathbf E} (D_m(T) \ | \ A_m)$ of magenta departures given the event $A_m$.  
Further, the event that there is 2 or more violet and/or magenta arrivals has probability $O(h^2)=o(h)$.
Since
\begin{align*}
& g(p)-g(p-h) = 
\frac{{\mathbf E} D_v(T)}{{\mathbf E} D(T)} =
ch \cdot \frac{{\mathbf E} (D_v(T)\ | A_v))}{{\mathbf E} D(T)} + o(h)\\
&=ch \cdot \frac{{\mathbf E} (D_m(T)\ | A_m))}{{\mathbf E} D(T)} + o(h) = \frac{{\mathbf E} D_m(T)}{{\mathbf E} D(T)} +o(h) =
g(p+h)-g(p) + o(h),
\end{align*}
we get
\begin{align*}
\frac{g(p)-g(p-h)}{h} = \frac{g(p+h)-g(p)}{h} + o(1).
\end{align*}
By letting $h$ tend to infinity, we obtain the desired differentiability of function $g$ at point $p$.
 
\paragraph{Proof of Lemma \ref{lem:unic}} 
The proof follows from the following arguments:\\
When $p=1$, we have $p_o<1$, and if $p=0$, then $p_o=0$. The function $p_o=g(p)$
is monotone increasing, strictly concave and differentiable. 
If the right derivative $g'(0)$ is less than or equal to 1, there is no other solutions
to $p=p_o$ within $(0,1)$.
If $g'(0)>1$, there is exactly one another solution, say $p^*$, to the fixed-point equation $p=g(p) \equiv p_o$, such that $0<p^*<1$.

\paragraph{Proof of Lemma \ref{lem8}}

Let $0<\widehat{\lambda}<\lambda$. By the Splitting theorem for Poisson processes, Poisson-$\widehat{\lambda}$
process may be obtained from Poisson-$\lambda$ process by i.i.d. thinning of points with acceptance probability $r=\widehat{\lambda}/\lambda$.

Consider the first busy cycle of length $T$ 
for the infinite-server queue with input rate $\lambda$, call it ``system 1''.
Let $D\equiv D(T)$ be the number of customers that are served in/departed from system 1
within time $[0,T]$. Then in ``system 2'' with input rate $\widehat{\lambda}$, the
total number $\widehat{D}$ of departures within $[0,T]$ has a conditional binomial 
distribution $Bin (D,r)$, i.e., given any value $D=k$, we have $\widehat{D} \sim Bin (k,r)$.
Then ${\mathbf E} \widehat{D} = r {\mathbf E} D$.

Let $D_I$ be the total number of infected departures from system 1 and
$\widehat{D}_I$ from system 2, within the time interval $[0,T]$. 
Clearly, $ D_I\ge \widehat{D}_I \ \ \mbox{a.s.}$

We have 
\begin{align*}
G(h,\lambda) =\frac{{\mathbf E} D_I}{{\mathbf E}D}
\ \  \mbox{and} \ \ 
G(h,\widehat{\lambda}) =\frac{{\mathbf E} \widehat{D}_I}{r{\mathbf E}D}.
\end{align*}

For small $h$, given $D$, the total number of arrivals of infected customers within $[0,T]$ to system 1
is 0 with probability $1-hD +o(h)$, 1 with probability $hD$, and more than one with probability $o(h)$. 
Let $\kappa_i=1$ if the $i$'th arrival is infected and $\kappa_i=0$,
otherwise (here ${\mathbf P} (\kappa_i=1)=h$).
Let $\zeta_i=1$ if the $i$'th arrival to system 1 is selected for system 2
and $\zeta_i=0$, otherwise (here ${\mathbf P}(\zeta_i=1)=r$). Then
\begin{align*}
{\mathbf E} D_I &= \sum_k {\mathbf P} (D=k)
\sum_{i=1}^k {\mathbf P} (A_{k,i})
{\mathbf E}\left(D_I \ | \ D=k, A_{k,i} \right)\\
&=
h  C_{\lambda} 
+o(h)
\end{align*}
where
\begin{align*}
C_{\lambda} =  \sum_k {\mathbf P} (D=k)
\sum_{i=1}^k 
{\mathbf E}\left(D_I \ | \ D=k, A_{k,i} \right)
\end{align*}
and, for $1\le i \le k$,
\begin{align*}
A_{k,i }= \{ \kappa_i=1, \kappa_j=0 \ \mbox{for all} \ 1\le j \le k, \ j\neq i\}.
\end{align*}
Next, 
\begin{align*}
{\mathbf E} \widehat{D}_I &= \sum_k {\mathbf P} (D=k) {\mathbf E}\left( \sum_{i=1}^k 
\widehat{D}_I {\mathbf I} (A_{k,i}){\mathbf I}(\zeta_i=1)
 \ | \ D=k\right) + o(h)\\
&\le 
 \sum_k {\mathbf P} (D=k) {\mathbf E}\left( \sum_{i=1}^k 
D_I {\mathbf I} (A_{k,i}){\mathbf I}(\zeta_i=1)
 \ | \ D=k\right) + o(h)\\
&= r {\mathbf E} D_I  + o(h),
\end{align*}
and we get the needed monotonicity.
In fact, there is strict inequality in the second line above since the event
\begin{align*}
\{ D=2, \kappa_1=1, \kappa_2=0, \zeta_1=0, \zeta_2=1\}
\end{align*}
has a positive probability and, assuming that the following sequence of events occur in system 1:
(1) customer 1 infects customer 2;
(2) customer 1 leaves;
(3) customer 2 leaves,
we get $D_I=2$ in the first system while $\widehat{D}_I=0$ in the second system.

\subsection{Monotonicity in $\alpha$ and $\beta$}
\label{ss:mab}

The construction of the coupling for the SIS reactor is as follows. The two processes are
piecewise constants and may change their states (jump) at embedded epochs $T_0=0<T_1<T_2<\ldots$ only.
Let $X_{i,n}=X_i(T_n+0)$ and $Y_{i,n}=Y_i(T_n+0)$,
for $i=1,2$ and $n=0,1,2,\ldots$. Further, let $N_n=X_{1,n}+Y_{i,n}$. 
At each time $T_n+0$, we set an exponential clock of rate $\lambda$, $N_n$ clocks of
rate $\mu$, $N_n$ clocks of rate $\beta$, and $N_n(N_n-1)$ clocks of rate $\alpha_2$.
We equip all customers, $i=1,\ldots,N_n$, with a $\mu$--clock and a $\beta$-clock, 
and all pairs of customers, say $(i,j)$.  with an $\alpha$-clock.
All clocks are mutually independent.

Let $T_{n+1}$ be the time when the first of these clocks rings. 
If this is the $\lambda$-clock, then a new customer arrives and it (simultaneously) becomes either $I$,
with probability $p$, or $S$, with probability $q=1-p$. If it is the $i$-th $\mu$-clock,
then the $i$'th customer leaves both systems simultaneously.
If it is the $i$'th $\beta$-clock, then the $i$'th customer becomes $S$ in both systems,
regardless of its earlier state. And if it is the $(i,j)$'th $\alpha$-clock, then, in the second system 
(that with infection parameter $\alpha_2$), the $j$'th customer
becomes $I$ if the $i$'th customer is $I$, regardless of the history, and does not change its state
if the $i$'th was $S$. In the first system (with infection parameter $\alpha_1$),
if the $i$'th customer is $I$, then the $j$'th customer becomes $I$ with probability
$\alpha_1/\alpha_2$, and keeps the earlier state (no jump) in all other cases.

with this coupling, direct induction arguments provide the announced monotonicity.

A similar coupling construction holds for the closed system with the following
simplifications and modifications:
here $N(t)\equiv N$, for any $t$; 
there is no $\lambda$-clock;
the ringing of a $\mu$-clock means that the corresponding customer
moves at random to any of $N$ stations, and we assume that in both systems, customers
move to the same stations;
when an $\alpha$-clock rings that corresponds to the $(i,j)$ pair of customers,
the $j$'th customer becomes $I$ if $i$ and $j$ are located at the same station.

\subsection{Negative Correlation and Anti-Association of the SIS Dynamics}
\label{sec4}

\subsubsection{Negative Correlation}

Below, we say that the RV $A$ is {\em more variant than Poisson} if
$\E[A^2]\ge \E[A]^2 +\E[A]$.

\begin{Lem}
\label{lem9}
In the contagion reactor,
let $X$ and $Y$ denote the the stationary number
of infected and susceptible customers, respectively.
If $X$ and $Y$ are both more variant than Poisson, then
$X$ and $Y$ are negatively correlated.
\end{Lem}

\begin{proof}
Since $X+Y$ is Poisson,
$$\E[(X+Y)^2]=(\E[X]+\E[Y])^2 +\E[X]+\E[Y].$$
Hence
$$2\E[XY]= 2\E[X]\E[Y] + (\E[X]^2+\E[X]-\E[X^2]) + 
(\E[Y]^2+\E[Y]-\E[Y^2]).$$
\end{proof}

\begin{Rem}
\label{rem1}
\begin{enumerate}
\item It follows from the proof of the preceding lemma that
negative correlation is equivalent to
\begin{equation}
(\E[X]^2+\E[X]-\E[X^2]) + (\E[Y]^2+\E[Y]-\E[Y^2])\le 0,
\end{equation}
which is weaker than having both $X$ and $Y$ more variant than Poisson.
\item If the more variable than Poisson assumption does not hold,
by using the fact that $\E[X^2]\ge \E[X]^2$ (in place of $\E[X^2]\ge \E[X]^2+\E[X]$)
and a similar inequality for $Y$,
we get that the following inequality always holds:
\begin{equation}
\label{eq:ersatz}
\E[XY] \le \E[X]\E[Y] + \frac 1 2 \E[X+Y],
\end{equation}
which is weaker than negative correlation.
\end{enumerate}
\end{Rem}

We recall that $R$ (resp. $I$) denote the stationary point process of recoveries (resp. infections)
with intensity $a_r$ (resp. $a_i$).

\begin{Lem}
\label{lem11}
In the stationary reactor
\begin{itemize}
\item[(i)] The random variable $X$ is more variable than Poisson iff
$$a_r \E_R^0[X^-]- a_i \E^0_I[X^+]\ge \E[X](\mu \E[X]-\lambda q).$$
\item[(ii)] The random variable $Y$ is more variable than Poisson iff
$$a_i \E_I^0[Y^-]- a_r \E^0_R[Y^+]\ge \E[Y]( \mu\E[Y]-\lambda p).$$
\item[(iii)] The random variables $X$ and $Y$ are negatively correlated iff
$$a_r \E_R^0[X^-] +a_i \E_I^0[Y^-] - a_r \E^0_R[Y^+]- a_i \E^0_I[X^+]
\ge  \E[X](\mu \E[X]-\lambda q) + \E[Y]( \mu\E[Y]-\lambda p) .$$
\end{itemize}
\end{Lem}

\begin{proof}
It follows from (\ref{eqsecx}) that
\begin{eqnarray*}
\mu(\E[X^2] - \E[X]^2 -\E[X]) & = & \lambda q \E[X]+ \beta \E[XY] -\alpha\E[XY(X-1)] -\mu\E[X]^2
\\ &= &
 \lambda q [X]+ a_r \E_R^0[X^-] -a_i\E_I^0[X^+]  -\mu\E[X]^2,
\end{eqnarray*}
which proves the first result.

Similarly, 
It follows from (\ref{eqsecy}) that
\begin{eqnarray*}
\mu \E[Y^2] - \mu \E[Y]^2 -\mu \E[Y] & = & \lambda p \E[Y] + \alpha  \E[XY^2)] -\beta \E[(Y-1)Y] - \mu \E[Y]^2\\
& = & a_i\E_I^0[Y^-] -  a_r\E^0_R[Y^+] + \E[Y](\lambda p -\E[Y]) ,
\end{eqnarray*}
which proves the first result.

From Remark \ref{rem1}, $X$ and $Y$ are negatively correlated iff
$$\E[X^2] - \E[X]^2 -\E[X]
+\E[Y^2] - \E[Y]^2 -\E[Y] \ge 0,$$
which proves the last result.
\end{proof}

\begin{Lem}
\label{lem:negcor}
Under the negative correlation conjecture, in the stationary reactor,
\begin{equation} \frac \lambda \mu \ge \E[X]\ge \frac{\mu +\beta+\alpha \frac \lambda \mu 
-\sqrt{(\mu +\beta+\alpha \frac \lambda \mu )^2 -4\alpha\lambda \left(q+\frac \beta \mu \right)}}{2\alpha}
\ge 0
\end{equation}
and
\begin{equation} 0 \le \E[Y]\ge \frac \lambda \mu -\frac{\mu +\beta+\alpha \frac \lambda \mu 
-\sqrt{(\mu +\beta+\alpha \frac \lambda \mu )^2 -4\alpha\lambda \left(q+\frac \beta \mu \right)}}{2\alpha}
\le \frac \lambda \mu.
\end{equation}
\end{Lem}
\begin{proof}
Using now the negative correlation conjecture, we get
\begin{equation}
\label{eq6}
\lambda q \le \mu \E[X] -\beta \E[Y] +\alpha \E[X]\E[Y],\end{equation}
and
\begin{equation} \lambda p \ge (\mu +\beta) \E[Y] -\alpha \E[X]\E[Y].\end{equation}
Using once more the fact that $\E[X]+\E[Y]=\frac \lambda \mu$, it is easily checked
that these two equations lead to the very same inequality, namely:
\begin{eqnarray}\label{eq:quadx}
\lambda q \le \mu \E[ X] -\beta \frac \lambda \mu +\beta \E[ X] +\alpha 
\frac \lambda \mu \E[ X] -\alpha \E[ X]^2.
\end{eqnarray}
Consider the last equation with equality. This quadratic in $\E[ X]$ 
has two positive roots, one larger than $\frac \lambda \mu$, and the other
smaller. The result then follows.
\end{proof}

\subsubsection{Anti-Association}
The anti association conjecture is that for all non-decreasing functions $f$ and $g$
from the integers to the real line, we have
\begin{equation}
\E [f(X)g(Y)] \le \E[f(X)]\E[g(Y)].
\end{equation}

\subsubsection{Equivalences}

The equivalences of Lemma \ref{lem11} can be rephrased as follows
\begin{Lem}
In the stationary thermodynamic limit
\begin{itemize}
\item[(i)] The random variable $X$ is more variable than Poisson iff
\begin{equation}
\label{eq-eqx1}
\E_R^0[X^-]\ge \E^0_I[X^+]\end{equation}
or equivalently iff
\begin{equation}
\E^0_I[X^+]\le \frac \alpha \beta.
\end{equation}
\item[(ii)] The random variable $Y$ is more variable than Poisson iff
\begin{equation}
\label{eq-eqy1}
\E_I^0[Y^-]\ge \E^0_R[Y^+]
\end{equation}
or equivalently iff
\begin{equation} \E^0_I[Y^-]\ge \E[Y]=p \frac{\lambda}{\mu}.
\end{equation}
\item[(iii)] The random variables $X$ and $Y$ are negatively correlated iff
\begin{equation}
\label{eq-eqgl1}
\E_R^0[X^- - Y^-] \ge \E^0_I[X^+-Y^+]
\end{equation}
or equivalently iff
\begin{equation}
\label{eq-eqgl2}
\frac \beta {\alpha +\beta}
\left(\E^0_I[Y^-] - p \frac \lambda \mu\right)
+ \frac{\beta}{\mu}
\left(\frac \beta \alpha -\E_I^0[X^+] \right)\ge 0.
\end{equation}
\end{itemize}
\end{Lem}
\begin{proof}
The first equivalence follows from (i) of Lemma \ref{lem11}.
The second equivalence follows from (\ref{eqsecxtherm1plus}).
The third equivalence follows from (ii) of Lemma \ref{lem11}.
The four-th equivalence follows from (\ref{eqsecytherm1plus}).
The fifth equivalence follows from (iii) of Lemma \ref{lem11}.
The last equivalence follows from the first item of Remark \ref{rem1} and from
(\ref{eqsecxtherm1plus}) and (\ref{eqsecytherm1plus}).
\end{proof}

We now study what happens under the negative correlation assumption.

\begin{Lem}
\label{lem:whunc}
Under the negative correlation conjecture and
the SIS thermodynamic propagation of chaos ansatz, if there is survival, then necessarily,
\begin{equation}\label{eq8}
\frac{\beta}{\alpha} < \frac{\lambda}{\mu}.\end{equation} 
In addition,
\begin{equation} \frac \beta \alpha \le \E[ X] =q^*\frac \lambda \mu ,\hspace{2cm}
\E[ Y]=p^* \frac \lambda \mu \le \frac \lambda \mu  - \frac \beta \alpha.\end{equation}
and
\begin{equation} \E[ XY]=p^* \frac {\lambda \beta} {\mu \alpha},\end{equation}
with
\begin{equation}\label{eq:fibo} 1-p^*= q^* \ge \frac{\mu \beta}{\lambda \alpha}.\end{equation} 

\end{Lem}
\begin{proof}
Under the foregoing assumption, $\E[X]= q^* \frac \lambda \mu$. Using this in 
(\ref{eq6}), we get that
\begin{equation}
\label{eq:byaxy}
\beta \E[Y] \le \alpha \E[ X]\E[ Y].
\end{equation}
If there is survival, then $\E[Y]>0$ so that
$\beta  \le \alpha \E[X],$ that is
$\beta  \le \alpha q^* \frac \lambda \mu.$
The assumption that $q^*<1$ then implies that
$\beta  < \alpha  \frac \lambda \mu.$
The other relations follow immediately.
\end{proof}
Note that by contraposition, if
$\frac{\beta}{\alpha} \ge \frac{\lambda}{\mu},$ then there is extinction.

\begin{Lem}
Whether or not the negative correlation property holds, one always has
\begin{equation}p^* \le \frac{\frac \lambda \mu - \frac \beta \alpha
+\sqrt{\left(\frac \lambda \mu - \frac \beta \alpha\right)^2+2\frac \lambda \mu }}{2\frac \lambda \mu}.
\end{equation}
\end{Lem}
\begin{proof}
By the same reasoning as in the proof of the last lemma, but
based on (\ref{eq:ersatz}) in place of (\ref{eq:byaxy}), one obtains that we always have
\begin{equation}
\label{eq:byaxynew}
\beta \E[Y] \le \alpha \left(\E[X]\E[Y] +\frac \lambda{2\mu}\right),
\end{equation}
or equivalently
\begin{equation*}
\label{eq:byaxynew+}
(p^*)^2 \frac \lambda \mu + p^* \left(\frac \beta \alpha - \frac {\lambda} \mu\right) -\frac 1 2 \le 0.
\end{equation*}
This immediately implies the bound stated in the lemma.
\end{proof}

\subsubsection{Anti-Association in the SIS Thermodynamic Limit}

Below, we simplify the notation and use $p$ (or $q=1-p$) in place of $p^*$ (or $q^*=1-p^*$).
When assuming anti-association,
the first equation of Corollary (\ref{cor1}) gives
\begin{equation}\label{eq:assoc2-1}
\left(\mu +p \frac {\lambda \alpha }{\mu}\right) \E[X^2]  \ge
\lambda \left(q +\frac \beta \mu\right)
\left(1 +\frac \beta \alpha\right)
+ q\frac \lambda \mu \left(\lambda q -\frac \beta \alpha(\mu+\alpha+\beta)\right).
\end{equation}
Using the fact that $\E[X^2]\le \frac{\lambda^2}{\mu^2} +\frac \lambda \mu$, this
in turn implies that $p$ satisfies
\begin{equation*}
\left(\mu +p \frac {\lambda \alpha }{\mu}\right) 
\left(\frac{\lambda^2}{\mu^2} +\frac \lambda \mu \right) 
\ge
\lambda \left(q +\frac \beta \mu\right)
\left(1 +\frac \beta \alpha\right)
+ q\frac \lambda \mu \left(\lambda q -\frac \beta \alpha(\mu+\alpha+\beta)\right).
\end{equation*}
The constant terms (in $p$) coincide in the LHS and the RHS of the last inequality. Hence
this boils down to
\begin{equation}
p \le \frac \mu \lambda
\left(\frac \alpha \mu 
\left(\frac{\lambda^2}{\mu^2} +\frac \lambda \mu \right) +\left(1+\frac \beta \alpha\right)
+ \frac 1 \mu \left(2 \lambda -\frac \beta \alpha(\mu+\alpha+\beta)\right)\right).
\end{equation}
This is not informative.
If rather than the bound 
$\E[X^2]\le \frac{\lambda^2}{\mu^2} +\frac \lambda \mu$, 
we use the (hypothetical) less than Poisson bound, namely
$\E[X^2]\le \frac{ q^2 \lambda^2}{\mu^2} +\frac {q\lambda} \mu$, 
then we get
\begin{equation*}
\left(\mu +p \frac {\lambda \alpha }{\mu}\right) 
\left(\frac{q^2 \lambda^2}{\mu^2} +\frac { q\lambda} \mu \right) 
\ge
\lambda \left(q +\frac \beta \mu\right)
\left(1 +\frac \beta \alpha\right)
+ q\frac \lambda \mu \left(\lambda q -\frac \beta \alpha(\mu+\alpha+\beta)\right).
\end{equation*}
This lead to a quadratic in $p$ which gives back the same bound as the one 
obtained by order 1.\\

The second equation of Corollary (\ref{cor1}) gives
\begin{eqnarray}\label{eq:assoc2-2}
\left( q\frac {\lambda \alpha }{\mu} -\mu -\beta \right) \E[X^2] & \le &
\left( q\frac {\lambda \alpha }{\mu} -\mu -\beta \right)
\left(\frac {\lambda^2}{\mu^2} +\frac \lambda \mu -
2p \frac {\lambda \beta}{\mu \alpha} \right)
+ p\frac \lambda \mu \left(\lambda p +\mu+\beta \right).
\end{eqnarray}
\begin{itemize}
\item Either $q=q^* \ge \frac{(\mu +\beta)\mu }{\alpha \lambda}$, and
then $p$ satisfies the following inequality:
\begin{eqnarray}
\lambda p\left(1-2 \frac \beta \mu\right) 
\ge \left(\mu+\beta \right) \left(2\frac {\beta}{\alpha} -1 \right) + 2\frac {\lambda \beta}{\mu},
\end{eqnarray}
when using the bound $\E[X^2] \le \frac {\lambda^2}{\mu^2}$.
In this case, it also directly follows from (\ref{eq:assoc2-1}) and (\ref{eq:assoc2-2}) that $p$ satisfies
\begin{eqnarray*}
& &\hspace{-2cm} \frac
{\lambda \left(q +\frac \beta \mu\right)
\left(1 +\frac \beta \alpha\right)
+ q\frac \lambda \mu \left(\lambda q -\frac \beta \alpha(\mu+\alpha+\beta)\right)}
{\mu +p \frac {\lambda \alpha }{\mu}}\\
& \le & 
\frac{
\left( q\frac {\lambda \alpha }{\mu} -\mu -\beta \right)
\left(\frac {\lambda^2}{\mu^2} +\frac \lambda \mu -
2p \frac {\lambda \beta}{\mu \alpha} \right)
+ p\frac \lambda \mu \left(\lambda p +\mu+\beta \right)
}
{ q\frac {\lambda \alpha }{\mu} -\mu -\beta}\ ,
\end{eqnarray*}
which is of degree 3 in the unknown $p$.
\item
Or $q=q^* < \frac{(\mu +\beta)\mu }{\alpha \lambda}$, and then $p$ satisfies the polynomial
inequality
\begin{eqnarray}
\left( q\frac {\lambda \alpha }{\mu} -\mu -\beta \right) \left(\frac {\lambda^2}{\mu^2}
+\frac \lambda \mu \right) & \le &
\left( q\frac {\lambda \alpha }{\mu} -\mu -\beta \right) 
\left(\frac {\lambda^2}{\mu^2} +\frac \lambda \mu
+2p \frac {\lambda \beta}{\mu \alpha} \right)
+ p\frac \lambda \mu \left(\lambda p +\mu+\beta \right).
\end{eqnarray}
In case the less than Poisson bound holds, we also have
\begin{eqnarray}
\left( q\frac {\lambda \alpha }{\mu} -\mu -\beta \right) \left(\frac {q^2\lambda^2}{\mu^2}
+\frac {q \lambda} \mu \right) & \le &
\left( q\frac {\lambda \alpha }{\mu} -\mu -\beta \right) 
\left(\frac {\lambda^2}{\mu^2} +\frac \lambda \mu
+2p \frac {\lambda \beta}{\mu \alpha} \right)
+ p\frac \lambda \mu \left(\lambda p +\mu+\beta \right).
\end{eqnarray}
\end{itemize}
these polynomial inequalities could help improving the bound in (\ref{eq:fibo}).

\subsection{The Thermodynamic Limits as Non-homogeneous Markov Processes}
\label{ss:nhmc}
In the transient thermodynamic limit of SIS, at time $t$, a station has a Poisson arrival
process of infected (resp. susceptible) customers with an intensity equal to $\mu \mathbf{E} [Y(t)]$
(resp. $\mu \mathbf{E}[X(t)]$, where $Y(t)$ (resp. $X(t)$) is the 
the number of infected (resp. susceptible) customers in the station at time $t$.
The aim of this section is to show that this can be
described in terms of a non-homogeneous Markov dynamics.

This is best explained when looking first at a discrete time version of the problem.
Time is slotted with increments of duration $h$.
At time $k=0$, choose an initial condition in the infinite-dimensional thermodynamic limit.
The latter consists in a product-form distribution of customers over stations,
with a fixed joint distribution ${\cal P}_0$ of infected and susceptible in any given station,
such that the sum of the two coordinates is Poisson $\eta$.
The number of infected (resp. susceptible) arrivals in time slot 1
is Poisson with parameter $h\mu \mathbf{E} [Y_0]$ (resp. $h\mu \mathbf{E} [X_0]$,
where the last expectations are w.r.t. ${\cal P}_0$.
In addition, each of the $Y_0$ infected customers tosses a 3 face
die with respective weights
$$1-\exp(-h(\beta+\mu)),\frac{\beta}{\beta+\mu}\exp(-h(\beta+\mu)), \frac{\beta}{\beta+\mu}\exp(-h(\beta+\mu)).$$
If the outcome is 1, it stays and keeps its SIS state.
If the outcome is 2, it stays and changes its SIS state.
If the outcome is 3, it leaves.
Similarly, conditionally on $Y_0$, each of the $X_0$ infected customers tosses a 3 face
die with respective weights
$$1-\exp(-h(\alpha Y_0+\mu)),
\frac {\alpha Y_0} {\alpha Y_0+\mu}\exp(-h(\alpha Y_0+\mu)),
\frac{\mu}{\alpha Y_0+\mu}\exp(-h(\alpha Y_0+\mu)).$$
If the outcome is 1, it stays and keeps its SIS state.
If the outcome is 2, it stays and changes its SIS state.
If the outcome is 3, it leaves.
All these define a state $(X_1,Y_1)$. Let ${\cal P}_1$ be the 
distribution of $(X_1,Y_1)$. 

The construction for all $k$ is then obtained by induction.
Assume the triple $(X_k,Y_k, {\cal P}_k)$ is well defined.
When applying the same dynamics to the last triple,
one gets at the same time a state $(X_{k+1},Y_{k+1})$ and
a distribution ${\cal P}_{k+1}$.

The key observations are then the following:
\begin{itemize}
\item Once the sequence of distributions $\{{\cal P}_k\}$ is determined, 
one can then see the evolution of $\{(X_k,Y_k)\}$ as that of a standard 
discrete-time discrete-space non-homogeneous Markov chain. 
\item A stationary distribution is simply a distribution $\cal P$
such that of ${\cal P}_0=\cal P$, then ${\cal P}_1=\cal P$.
\item By tightness arguments, this should be extended to the
continuous time case when letting $h\to 0$. 
\end{itemize}

\subsection{Proof of Theorem \ref{thm:ext}}
\label{ss:pttext}

Below, $p^*$ is the maximum root of the fixed-point equation $p=g(p)$ for the plain SIS reactor.
This means that if $g^{'}(0)\le 1$, $p^*=0$, whereas if $g^{'}(0)>1$, 
$p^*$ is positive and there are two solutions to the fixed-point equation.

Consider an $M/M/\infty$ system in the stationary regime on the time horizon $[0,\infty)$,
with input rate $\lambda$ and service rate $\mu$.
The number of customers in the system at any time $t$, $N(t)$, has a Poisson distribution
with parameter $\eta = \lambda/\mu$.

\paragraph{Monotonicity}

Consider the SIS reactor with a Poisson input point process with a time varying intensity
as that discussed in Remark \ref{rem24}.
Consider two variants of customer ``coloring'' at time $0$ and of input rates.
In the first variant, there are initially $Y(0)$ infected customers and the input point process
of infected customers has an intensity equal to $\lambda p(.)$, while in the second variant there are
initially $\widehat Y(0)$ infected customers and the intensity of infected customers
is $\lambda \widehat p(.)$. Here $p(t)$ and $\widehat p(t)$ are two given deterministic functions. 

The next result follows from the monotonicity property
established in Subsection \ref{secap1} - see more precisely Remark \ref{rem24}.

\begin{Lem}\label{monotX}
If $Y(0)\le\widehat Y(0)$ a.s. and $p(.)\le\widehat p(.)$, then, under an appropriate coupling,
\begin{align*}
Y(t)\le \widehat Y(t)\quad \forall t,\quad a.s.
\end{align*}
\end{Lem}
In what follows, ``by simple monotonicity arguments'' means ``by applying Lemma \ref{monotX}''.

\paragraph{The Upper Process}

We introduce a SIS model with a varying fraction of infected customers
in the input process and instantaneous extra infections along the lines
described in Remark \ref{rem24}.

Recall that we consider a total population dynamics which is that of an $M/M/\infty$
system in the stationary regime on the time horizon $[0,\infty)$, with input
rate $\lambda$ and service rate $\mu$. 

Consider now the following infection/recovery mechanism.
We introduce recursively deterministic times $T^{(0)}=0<T^{(1)}<T^{(2)}<\ldots $
and probabilities $p^{(1)}=1>p^{(2)}> \ldots$ such that
\begin{itemize}
\item{(i)} $p^{(n)}\downarrow p^*$ as $n\to\infty$;
\item{(ii)} for any $n=1,2,\ldots$, each customer that arrives to the system within time interval
$(T^{(n-1)},T^{(n)})$ is infected with probability $p^{(n)}$ and susceptible, otherwise,
independently of everything else;
\item{(iii)} for any $n=1,2,\ldots$, at time $T^{(n-1)}$,
all customers present in the system become infected instantaneously. 
\end{itemize}

\noindent
In more detail, with $n=1,2,\ldots$, we introduce recursively processes 
$(\widehat{X}^{(n)}(t),\widehat{Y}^{(n)}(t))$, for $t\ge T^{(n-1)}$, and then let 
\begin{align}\label{ququ}
(\widehat{X}(t),\widehat{Y}(t)) =
(\widehat{X}^{(n)}(t),\widehat{Y}^{(n)}(t)) \quad \text{for} \quad t\in [T^{(n-1)}, T^{(n)}).
\end{align}
First, we assume that, at time $T^{(0)}=0$, all customers that are present
in the system are infected and that, starting from time 0, all arriving customers are infected too 
(i.e., each is infected with probability $p^{(1)}=1$).
We have a time-homogeneous and irreducible Markov jump process on a countable state space, which is clearly aperiodic and positive recurrent.
Therefore, there exists a unique stationary distribution which is the limit in
law obtained for any initial state, and the speed of convergence to this stationary distribution
is exponential. In the stationary regime, the output rate of infected customers
is $p_o^{(1)}$, which is strictly less than $p^{(1)}=1$.
In particular, if one denotes by $(\widehat{X}^{(1)}(t),\widehat{Y}^{(1)}(t))$ the
state process in this dynamics, one has ${\mathbf E} \widehat{Y}^{(1)}(t) \to p^{(1)}_o \eta$
exponentially fast, as $t\to\infty$. 

Let $\delta \in (0,1)$ be arbitrarily small. We choose time $T^{(1)}$ as 
\begin{align}\label{slot1}
T^{(1)} = \min \{t\ge T^{(0)}: {\mathbf E} Y(u)/\eta \le 
p_o^{(1)}+ \delta (p^{(1)}-p^{(1)}_o)/2, \ \text{for all} \ u\ge t\}.
\end{align}
Then let $p^{(2)}=p^{(1)}_o+\delta (p^{(1)}-p^{(1)}_o)/2$.
By \eqref{ququ}, this ends the description of the dynamics of the process $(\widehat{X}(t), \widehat{Y}(t))$ within the first time
slot $[T^{(0)},T^{(1)})$. Note that, due to monotonicity and convexity of the function
$p_o = g(p)$, we get $p^{(1)}_0>p^*$.

Assume we have introduced the processes $(\widehat{X}^{(i)}(t),\widehat{Y}^{(i)}(t))$
on the time intervals $[T^{(i-1)},\infty)$, for $i=1,\ldots,n-1$, and, therefore,
have defined the process $(\widehat{X}(t), \widehat{Y}(t))$ on the interval $[0,T^{(n-1)})$ via \eqref{ququ}. 
At the beginning of the $n$-th time slot (at time $T^{(n-1)}$), we turn all customers present in the system
to infected and assume that, starting from this time, each arriving customer is either infected,
with probability $p^{(n)}>p^*$, or susceptible, otherwise. Then, starting from time $T^{(n-1)}$,
our system again behaves as an irreducible time-homogeneous Markov jump process whose
distribution converges exponentially fast to its unique stationary distribution
with output fraction of infected customers $p^{(n)}_o $, which is strictly bigger
than $p^*$ (thanks again to monotonicity and convexity of the function $g$).
We denote this process by $(\widehat{X}^{(n)}(t),\widehat{Y}^{(n)}(t)), t\ge T^{(n-1)}$. Then we let 
\begin{align}\label{slotn}
T^{(n)} = \min \{t\ge T^{(n-1)}: {\mathbf E} \widehat{Y}^{(n)}(u)/\eta \le 
p_o^{(n)}+\delta (p^{(n)}-p^{(n)}_o)\cdot 2^{-n}, \ \text{for all} \ u\ge t\}
\end{align}
and $p^{(n+1)}=p^{(n)}_o+\delta (p^{(n)}-p^{(n)}_o)\cdot 2^{-n}$.
This ends the description of the process $(\widehat{X}(t),\widehat{Y}(t))$ within the $n$-th time
slot $[T^{(n-1)},T^{(n)})$. Again, thanks to monotonicity and convexity of the function
$p_o = g(p)$, we get $p^{(n)}_o> p^{(n+1)}_o>p^*$.

Thus, we have defined the process $(\widehat{X}(t),\widehat{Y}(t)$ for all $t\ge 0$. It is left to explain the convergence (i).
We know that there is no solutions to equation $p_0=g(p)$ bigger than $p^*$,
and the monotone sequences $p^{(n)}$ and $p^{(n)}_o$ converge to the same limit
that is a solution to this fixed-point equation. The fact that this is the maximal solution
proves the result.

\paragraph{Comparison of the Upper Process and the Thermodynamic Limit Process}

As shown in Subsection \ref{ss:nhmc}, under the propagation of chaos ansatz,
the dynamics of any station in the SIS thermodynamic limit can be seen
as a non-homogeneous Markov jump process $(X(t),Y(t))$ with a certain fraction $p^*(t)$ of infected customers
at time $t$ (the fraction such that at any time $t$, $\lambda p^*(t)=\mu \mathbf{E} Y(t)$).

By construction,
\begin{align*}
{X}(t)+{Y}(t)=N(t), \quad \text{for all} \quad t\ge 0,
\end{align*}
where $N(t)$ is the Poisson process describing the dynamics of the stationary
$M/M/\infty$ system and also 
$N(t)= \widehat{X}(t)+\widehat{Y}(t)$ a.s. (where the latter are the
state variables in the upper-system described above).

\begin{Prop}
\label{lem:lastl}
For all $t\ge 0$, in a proper coupling,
\begin{align}\label{mainin}
Y(t) \le \widehat{Y}(t) \quad \text{a.s.}
\end{align}
and, therefore,
\begin{align}
\limsup_{t\to\infty}
{\mathbf E} Y(t) \le p^*.
\end{align}
\end{Prop}

\begin{proof}
First, by simple monotonicity arguments,
\begin{align*}
Y(t) \le \widehat{Y}^{(1)}(t) \quad \text{a.s., for all} \quad t\ge 0.
\end{align*}
In particular, 
\begin{align*}
{\mathbf E} Y(t)/\eta \le  p^{(2)}, \quad
\text{for all}\quad t\ge T^{(1)}.
\end{align*}
Consider now the time interval $[T^{(1)}, \infty)$ and compare
the processes $(\widehat{X}^{(2)}(t),\widehat{Y}^{(2)}(t))$ and $(X(t),Y(t))$ within this interval. 
By construction,
at the initial time $T^{(1)}$ of this time period, 
$\widehat{Y}^{(2)}(T^{(1)})\ge Y(T^{(1)})$ a.s. and,
for all $t\ge T^{(1)}$, the input rate of infected customers in the auxiliary upper process is
bigger than that in the thermodynamic limit.
Then it follows from simple sample-path arguments that
$Y(t)\le \widehat{Y}^{(2)}(t)$ a.s. for all $t\ge T^{(1)}$ and, in particular,
${\mathbf E}(Y(t)) \le {\mathbf E} \widehat{Y}^{(2)}(t)$, where the RHS is smaller than $p_1^{(3)}$, for
all $t\ge T^{(3)}$.

The induction argument implies that, for any $n=1,2,\ldots$ and
any $k\le n$ and $t\ge T^{(n-1)}$, the following inequalities hold a.s.:
\begin{align*}
Y(t) \le \widehat{Y}^{(k)}(t) \quad \text{and, therefore,} \quad
Y(t) \le \widehat Y(t).
\end{align*}
It is left to show that $\lim p^{(k)}=p^*$. However, this is clear since $p^*$
is the maximal solution to the equation $p=g(p)$ and, due to convexity,
all $p^{(k)}>p^*$ since $p^{(k)}_0 = g(p^{(k)})<p^{(k)}$.
\end{proof}
\begin{Cor}\label{Extinction}
If $g^{'}(0)\le 1$, then there is extinction of the infection process in that
\begin{align}\label{ext}
\lim_{t\to\infty} {\mathbf E} {Y}(t)=0,
\end{align}
for all initial distributions of infected customers in the system.
\end{Cor}

The next Corollary follows fron simple monotonicity arguments. 
Before stating it, we recall the following
result on the fixed-point of the SIS Reactor. Assume $g^{'}(0)>0$, which implies that
$p^*>0$. Any SIS Reactor is a time-homogeneous Markov process
$(\widetilde X(t),\widetilde Y(t))$
with a countable state space, which is irreducible, positive recurrent, and ergodic.
In particular, for $p=p^*$, 
for all initial conditions $(\widetilde X(0),\widetilde Y(0))$, this Markov process 
converges to a limiting stationary Markov process with $\mu {\mathbf E} \widetilde Y= p^*\lambda$.
Denote this last stationary process by $\{(\widetilde X^*(t),\widetilde Y^*(t))\}_t$.
This last process can be used to build a specific instance 
$(X^*(t), Y^*(t),{\cal P}^*_t)$, $t\ge 0$,
of the non-homogeneous Markov chain of Subsection \ref{ss:nhmc},
where the initial distribution ${\cal P}^*_0$ is that of $(\widetilde X^*(0),\widetilde Y^*(0))$.
This distribution is {\em invariant} for this non-homogeneous dynamics. By this, we mean 
that it makes the Markov chain homogeneous {\em and} stationary, with an
input rate of infected which is constant, equal to $\lambda p^*$, and equal to the mean output rate of infected.

\begin{Cor}\label{Squeeze}{\ }
\begin{itemize}
\item{}
Consider the process $Y^*(t)$ defined above. Then
\begin{align*}
\widehat{Y}(t)\ge Y^*(t) \ \text{ a.s. for all} \ t\ge 0 \ \text{and} \ \ {\mathbf P}(\widehat{Y}(t)=Y^*(t)) \to 1, \ \ \text{as} \ \ t\to\infty,
\end{align*}
where $\widehat{Y}(t)$ is the the upper process of Proposition \ref{lem:lastl}.
\item{}
Consider the non-homogeneous Markov process ${Y}(t)$ defined in Subsection
\ref{ss:nhmc} with any initial distribution such that
${Y}(0) \ge Y^*(0)$ a.s., with $Y^*(.)$ defined above. Then
\begin{align}\label{ququ2}
\widehat{Y}(t)\ge {Y}(t)\ge Y^*(t) \quad \text{a.s. and} \quad
{\mathbf P} (\widehat{Y}(t)= {Y}(t) = Y^*(t)) \to 1. 
\end{align}
\end{itemize}
\end{Cor}

The only thing to comment is the second formula in \eqref{ququ2}. It follows from the
facts that all $Y$-variables are integer-valued and that the monotone convergence
of integer-valued random variables is necessarily also a coupling (or total variation)convergence.

\section{Acknowledgements}
The work of the first author was supported by
by the ERC NEMO grant, under the European Union's Horizon 2020 research and innovation programme,
grant agreement number 788851 to INRIA.

\end{document}